\documentclass[a4paper,11pt,UKenglish]{article}

\usepackage[top=3.2cm, bottom=3cm, left=3.45cm, right=3.45cm]{geometry}
\usepackage[T1]{fontenc}
\usepackage[utf8]{inputenc}
\usepackage{latexsym}
\usepackage{amsmath,amsthm,amssymb,amsfonts}
\usepackage{mathtools}
\usepackage[english]{babel}
\usepackage{tikz}
\usetikzlibrary{patterns, matrix}
\tikzstyle{NE-lines}=[pattern=north east lines, pattern color=black!45]

\def\loongmapsto#1{%
  \begin{tikzpicture}
    \draw (0,0.5mm) -- (0,-0.5mm);
    \draw[->] (0,0) -- (1.4, 0) node[above,midway] {\scriptsize #1};
  \end{tikzpicture}
}
\newcommand{\etal}{et~al.}
\newcommand{\Ascseq}{A}                   
\newcommand{\Modasc}{\hat{\Ascseq}}       
\newcommand{\BurMA}{\hat{\mathcal{A}}}    
\newcommand{\F}{F}                        
\newcommand{\xset}{X}
\newcommand{\BurF}{\mathcal{F}}           
\newcommand{\Cay}{\mathrm{Cay}}           
\newcommand{\BurCay}{\mathcal{C\kern-0.9pt ay}} 
\newcommand{\Bur}{\mathrm{Bur}}           
\newcommand{\Sym}{S}                      
\newcommand{\RGF}{\mathrm{RGF}}           
\newcommand{\WI}{I}                       

\newcommand{\Asc}{\textrm{Asc}}           
\newcommand{\Des}{D}                      
\newcommand{\asc}{\mathrm{asc}}           
\newcommand{\des}{\mathrm{des}}           
\newcommand{\fun}{\varphi}                
\newcommand{\hatfun}{\psi}                
\newcommand{\Mat}{M}                      
\newcommand{\identity}{\mathrm{id}}       
\newcommand{\fishpattern}{\mathfrak{f}}   
\newcommand{\basis}{B}                    
\newcommand{\cbasis}{C}                   
\newcommand{\sort}{\mathrm{sort}}         
\newcommand{\asctops}{\mathrm{top}}       
\newcommand{\nub}{\mathrm{nub}}           

\newcommand{\hatx}{\hat{x}}

\newcommand{\nyip}{\hspace*{0.425em}}
\mathchardef\mhyphen="2D
\newcommand{\dashpatt}{23\mhyphen 1}
\newcommand{\rgfpatt}{\mathfrak{g}}
\newcommand{\Flat}{\mathrm{Flat}}      

\newtheorem{theorem}{Theorem}[section]
\newtheorem{theorem*}{Theorem}[section]
\newtheorem{prop}[theorem]{Proposition}
\newtheorem{lemma}[theorem]{Lemma}
\newtheorem{corollary}[theorem]{Corollary}
\newtheorem{openproblem}[theorem]{Open Problem}
\newtheorem*{openproblem*}{Open Problem}

\theoremstyle{definition}
\newtheorem{defin}[theorem]{Definition}
\newtheorem{remark}[theorem]{Remark}
\newtheorem*{remark*}{Remark}

\newtheorem*{example*}{Example}

\title{Transport of patterns by Burge transpose}

\author{Giulio Cerbai\footnote{The author is a member of the INdAM
    Research group GNCS; he is partially supported by the INdAM - GNCS 2020
    project ``Combinatoria delle permutazioni, delle parole e dei grafi:
    algoritmi e applicazioni".} \and Anders
    Claesson\footnote{This material is based upon work supported by the
    Swedish Research Council under grant no. 2016-06596 while the author
    was in residence at Institut Mittag-Leffler in Djursholm, Sweden
    during Spring 2020.}}

\begin{document}
\maketitle

\begin{abstract}
  We take the first steps in developing a theory of transport of
  patterns from Fishburn permutations to (modified) ascent
  sequences. Given a set of pattern avoiding Fishburn permutations, we
  provide an explicit construction for the basis of the corresponding
  set of modified ascent sequences.  Our approach is in fact more
  general and can transport patterns between permutations and
  equivalence classes of so called Cayley permutations.  This transport
  of patterns relies on a simple operation we call the Burge
  transpose. It operates on certain biwords called Burge words. Moreover,
  using mesh patterns on Cayley permutations, we present an alternative
  view of the transport of patterns as a Wilf-equivalence between
  subsets of Cayley permutations. We also highlight a connection
  with primitive ascent sequences.\smallskip

  \noindent \textit{Keywords:} Fishburn permutation, Cayley permutation,
  Burge word, transpose, ascent sequence, pattern avoidance.
\end{abstract}

\section{Introduction}
\thispagestyle{empty}

In 2010 Bousquet-Mélou, Claesson, Dukes and Kitaev~\cite{BMCDK}
introduced ascent sequences, which they used as an auxiliary set of
objects that most transparently embodies the recursive structure that
they discovered on $(2+2)$-free posets, Stoimenow's matchings and a set
of pattern avoiding permutations, now called Fishburn permutations. All
of these objects are enumerated by the Fishburn numbers, which is
sequence A022493 in the OEIS~\cite{Sl}. This counting sequence has a
beautiful generating function~\cite{Sl,Za}:
$$\sum_{n\geq 0}\prod_{k=1}^n\bigl(1-(1-x)^k\bigr)
= 1 + x + 2x^2 + 5x^3 + 15x^4 + 53x^5 + 217x^6 + \cdots
$$
Since then, ascent sequences have been studied in
their own right. In particular, pattern avoiding ascent sequences have
been quite thoroughly investigated~\cite{BP,CMS,DS,MS14,MS15}. The study
of pattern avoidance on ascent sequences has proved itself to often be
even more intricate than its analogue on permutations and a framework
capable of producing general results is missing.

Recently, Gil and Weiner~\cite{GW} studied pattern avoidance on Fishburn
permutations.  The main purpose of this work is to initiate the
development of a theory of transport of patterns from Fishburn
permutations to ascent sequences, and vice versa, aiming towards a more
general understanding of pattern avoidance. Instead of ascent sequences
we use their modified version~\cite{BMCDK}. The main benefit is that
permutations as well as modified ascent sequences are Cayley
permutations, and they provide a natural setting for the transport of
patterns. The necessary background on Cayley permutations and pattern
avoidance is given in Section~\ref{Section_preliminaries}.

In Section~\ref{Section_Burge_words} we introduce the Burge transpose of
biwords. This operation provides a high-level description of a bijection
$\psi$ between modified ascent sequences and Fishburn permutations originally
given by Bousquet-Mélou~\etal~\cite{BMCDK}. In
Section~\ref{Section_transport_theorem} we use the Burge transpose to
define an equivalence relation on Cayley permutations and to equip its
equivalence classes with a notion of pattern avoidance. The avoidance of
a pattern on the quotient set is transported by Burge transposition to
classical pattern avoidance on permutations, thus yielding a
general result on the transport of patterns. This machinery can be specialized
by suitably choosing representatives for the equivalence classes. The most
striking example is a transport theorem for Fishburn
permutations and modified ascent sequences, which we
prove in Section~\ref{Section_transport_F_Modasc}:
\textit{Given a pattern $\sigma$ we describe an explicit construction of
a set of patterns $B(\sigma)$ such that Fishburn permutations avoiding
$\sigma$ are in one-to-one correspondence, via the Burge transpose, with
modified ascent sequences avoiding all of the patterns in $B(\sigma)$.}

In Section~\ref{Section_Examples}, the same construction will be extensively used
to derive a number of structural and enumerative results that link pattern
avoiding (modified) ascent sequences to the corresponding Fishburn permutations.
Table~\ref{table_examples_transport} contains several examples that
illustrate this approach. As a corollary of the same framework, we also
obtain a transport theorem for restricted growth functions and permutations
avoiding the vincular pattern $\dashpatt$.

In Section~\ref{Section_lift_of_psi} we ``lift'' the mapping
$\psi^{-1}$, whose domain is the set of Fishburn permutations, to a
new mapping $\eta$ whose domains is $\Sym$, the set of all
permutations. The map $\eta$ encodes what we call the $\eta$-active
sites of a permutation. In particular, $\eta$ preserves the property of
transporting patterns, thus generalizing the transport theorem for Fishburn
permutations. We then characterize the image set $\eta(\Sym)$
in terms of mesh patterns on Cayley permutations. This further allows us
to characterize modified ascent sequences as pattern avoiding Cayley
permutations, a noteworthy consequence of which is that the transport of
patterns can be regarded as a theory of Wilf-equivalence on Cayley
permutations. We close Section~\ref{Section_lift_of_psi} by studying
the set $\eta(\Sym)\cap\Sym$. This set can be described as the image
under $\eta$ of the set of permutations in which all sites are
$\eta$-active, which in turn is shown to be in bijection with primitive
ascent sequences.

In Section~\ref{Section_future_works} we raise some natural
questions, leaving two of them as open problems.

\section{Preliminaries}\label{Section_preliminaries}
\subsection{Cayley permutations and pattern avoidance}\label{Section_Cayley_perms}

A word consisting of positive integers that include at least one copy of
each integer between one and its maximum value is called a \emph{Cayley
permutation}~\cite{Ca, MF84}. We will denote by $\Cay_n$ the set of Cayley
permutations of length $n$,
and by $\Cay=\cup_{n\geq 0}\,\Cay_n$ the set of Cayley permutations
of any finite length.
For example, $\Cay_1=\{1\}$,
$\Cay_2 = \{11, 12, 21\}$ and
$$
\Cay_3 = \{111 ,112 ,121 ,122 ,123 ,132 ,211 ,212 ,213 ,221 ,231 ,312 ,321\}.
$$
Equivalently, a word $x_1x_2\dots x_n$ belongs to $\Cay_n$ precisely when
there is an endofunction $x:[n]\to [n]$ such that $\mathrm{Im}(x) = [k]$ for
some $k\leq n$ and $x(i)=x_i$ for each $i$ in $[n]$. We can also view
$x$ as encoding a ballot (ordered set partition) with blocks
$B_1B_2\dots B_k$ such that $i\in B_{x(i)}$. Thus, the cardinality of
$\Cay_n$ is the $n$th Fubini number, which is sequence A000670 in the
OEIS~\cite{Sl}.

A bijective endofunction $\pi:[n]\to [n]$ is called a \emph{permutation}
and $n$ is said to be the length of $\pi$.  We shall sometimes write
permutations in so called one-line notation and thus identify $\pi$ with
its list of images $\pi(1)\pi(2)\cdots \pi(n)$. We will denote by
$\identity_n$ the identity permutation, $\identity_n(i)=i$, in $S_n$.
In fact, we shall often just write $\identity$ (without the subscript)
and let $n$ be inferred by context. Denote by $\Sym_n$ the
set of permutations of length $n$ and by $\Sym=\cup_{n\geq 0}\,\Sym_n$
the set of permutations of any finite length. Note that
$\Sym \subseteq \Cay$.

Given two Cayley permutations $x$ and $y$, we say that $y$ is a
\emph{pattern} in $x$ if $x$ contains a subsequence
$x(i_1)x(i_2)\cdots x(i_k)$, with $i_1 \le i_2 \le \cdots \le i_k$,
which is order isomorphic to $y$; that is, $x(i_s) < x(i_t)$ if and only
if $y(s) < y(t)$ and $x(i_s) = x(i_t)$ if and only if $y(s) = y(t)$.
In this case we write $y\leq x$ and $x(i_1)x(i_2)\cdots x(i_k)\simeq y$; the subsequence $x(i_1)x(i_2)\cdots x(i_k)$ is called an
\emph{occurrence} of $y$ in $x$. Otherwise, $x$ \emph{avoids}
$y$. Denote by $\Cay(y)$ the set of Cayley permutations that avoid $y$
and by $\Cay_n(y)$ the set $\Cay(y) \cap \Cay_n$ of Cayley permutations
of length $n$ avoiding $y$. For example, $\Sym=\Cay(11)$ is the set of
permutations. If $B$ is a set of patterns, $\Cay(B)$ denotes the set of
Cayley permutations avoiding every pattern in $B$ and $\Cay_n(B)$
denotes $\Cay_n \cap \Cay(B)$. We use analogous notations for subsets of
$\Cay$. For instance, $\Modasc(212,312)$ denotes the set of modified
ascent sequences (defined in Section~\ref{Section_asc_seq}) avoiding the
two patterns $212$ and $312$. The containment relation is a partial
order on $\Sym$ and downsets in this poset are called \emph{permutation
  classes}. Similarly, the containment relation is a partial order on
$\Cay$ and downsets in this poset are called \emph{Cayley permutation
  classes}. The \emph{basis} of a (Cayley) permutation class is the
minimal set of (Cayley) permutations it avoids. For instance, the basis
for $S$ in $\Cay$ is $\{11\}$.  For a more detailed introduction to
permutation patterns we refer the reader to Bevan's note ``Permutation
patterns: basic definitions and notations''~\cite{Be}.

The set $\Sym$ of permutations can be equipped with more general notions
of patterns~\cite{BS,BMCDK,BC,Cl}. A \emph{bivincular
  pattern}~\cite{BMCDK} of length $k$ is a triple $(\sigma,X,Y)$, where $X$
and $Y$ are subsets of $\lbrace 0,1,\dots,k \rbrace$ and
$\sigma\in\Sym_k$. An occurrence of $(\sigma,X,Y)$ in a permutation $\pi\in\Sym_n$
is then an occurrence $\pi(i_1) \cdots \pi(i_k)$ of $\sigma$ (in the classical
sense) such that:
\begin{itemize}
\item $i_{\ell+1}=i_\ell+1$, for each $\ell \in X$;
\item $j_{\ell+1}=j_{\ell}+1$, for each $\ell \in Y$,
\end{itemize}
where
$\lbrace \pi(i_1),\dots,\pi(i_k) \rbrace = \lbrace j_1,\dots,j_k \rbrace$,
with $j_1<\cdots<j_k$; by convention, $i_0=j_0=0$ and
$i_{k+1}=j_{k+1}=n+1$. The set $X$ identifies constraints of adjacency on
the positions of the elements $\pi$, while the set $Y$, symmetrically,
identifies constraints on their values. An example of a bivincular pattern
is depicted in Figure~\ref{mesh_fishburn}. If $Y$ is empty, then 
$(\sigma,X,Y)$ is called a \emph{vincular} pattern.

By allowing more general constraints on positions and values we
arrive at mesh patterns. A \emph{mesh pattern}~\cite{BC} is a pair $(\sigma,R)$,
where $\sigma \in \Sym_k$ is a permutation (classical pattern) and
$R \subseteq \left[ 0,k \right] \times \left[ 0,k \right]$ is a set of
pairs of integers. The pairs in $R$ identify the lower left corners of
unit squares in the plot of $\pi$ which specify forbidden regions. An
occurrence of the mesh pattern $(\sigma,R)$ in the permutation $\pi$ is an
occurrence of the classical pattern $\sigma$ such that no other points of the
permutation occur in the forbidden regions specified by $R$.

Two subsets of $\Cay$ are \emph{equinumerous} if they contain the same
number of Cayley permutations of each length. Equivalently, if they have
the same generating function. Two sets of (generalized) patterns $B_1$
and $B_2$ are \emph{Wilf-equivalent} if $\Sym(B_1)$ and $\Sym(B_2)$
are equinumerous. We extend this notion to Cayley permutations by saying
that $B_1$ and $B_2$ are Wilf-equivalent (over $\Cay$) if $\Cay(B_1)$
and $\Cay(B_2)$ are equinumerous.

\subsection{Ascent sequences}\label{Section_asc_seq}

Let $x:[n]\to [n]$ be an endofunction. We call $i\in [n-1]$ an
\emph{ascent} of $x$ if $x(i) < x(i+1)$.  Let $\asc(x)$ denote
the number of ascents of $x$. Then $x$ is an \emph{ascent sequence}
of length $n$ if $x(1)=1$ and
$x(i+1) \le 2+\asc\bigl(x\circ \identity_{i,n}\bigr)$ for each
$i\in [n-1]$, where $\identity_{i,n}:[i] \to [n]$ is the inclusion
map. Let $\Ascseq_n$ be the set of ascent sequences of length $n$. For
instance, $\Ascseq_3 = \{111, 112, 121, 122, 123\}$. Note that some
ascent sequences are not Cayley permutations, the smallest example of
which is $12124$. Note also that we depart slightly from the original
definition of ascent sequences~\cite{BMCDK} in that our sequences are
one-based rather then zero-based. The reason for this is that we want to
bring all the families of sequences considered in this paper under one
umbrella, namely that of endofunctions on $[n]$.

We shall now define the set of \emph{modified ascent
sequences}~\cite{BMCDK}, denoted $\Modasc_n$. This set,
which is equinumerous with $\Ascseq_n$, has a recursive
structure that is similar to, but more complicated than, that of
$\Ascseq_n$. The definition goes as follows. There is exactly one
modified ascent sequence of length zero, namely the empty word. There is
also exactly one modified ascent sequence of unit length, namely the
single letter word $1$. Suppose $n\geq 2$. Every $x\in\Modasc_n$ is of
one of two forms depending on whether the last letter forms an ascent
with the penultimate letter:
\begin{itemize}
\item $x = \upsilon a$ \,and\, $1\leq a \leq b$, or
\item $x = \tilde{\upsilon} a$ \,and\, $b< a \leq 2+\asc(\upsilon)$,
\end{itemize}
where $\upsilon\in\Modasc_{n-1}$, the last letter of $\upsilon$ is $b$,
and $\tilde{\upsilon}$ is obtained from $\upsilon$ by increasing each
entry $c \geq a$ by one.

\begin{lemma}\label{lemma_ascent_tops}
  Let $x\in\Modasc_n$ be a modified ascent sequence. An element $x(i)=k>1$
  is the leftmost occurrence of the integer $k$ in $x$ if and only if
  $x(i-1)<x(i)$. In particular, $x$ is a Cayley permutation.
\end{lemma}
\begin{proof}
  We proceed by induction on the length of $x$, using the recursive
  definition of $\Modasc$. If
  $n=0$ or $n=1$, then there is nothing to prove. Suppose $n \ge 2$ and let
  $x\in\Modasc$. Let $a = x(n)$. We have either
  \begin{itemize}
  \item $x=va$ and $1 \le a \le b$, or
  \item $x=\tilde{v}a$ and $b<a\le 2+\asc(v)$,
  \end{itemize}
  where $v \in \Modasc_{n-1}$, the last letter of $v$ is $b$ and
  $\tilde{v}$ is obtained from $v$ by increasing each entry $c \ge a$ by
  one.
  Note that in both cases $x$ is a Cayley permutation since $v\in\Cay$ by
  the inductive hypothesis.
  Moreover, the desired property holds for $v$ (again, by the inductive
  hypothesis) and it holds for $\tilde{v}$ as well since increasing each
  element greater than or equal to a certain value by one preserves it.
  Now, if $1\le a\le b$, then $v$ already contains an
  occurrence of $a$ (since $v$ is a Cayley permutation) and therefore
  $x(n)$ is not the leftmost occurrence of $a$ in $x$. Finally, if
  $b<a\le 2+\asc(v)$, then $x(n)$ is the only (and thus leftmost)
  occurrence of $a$ in $x$.
\end{proof}

By Lemma~\ref{lemma_ascent_tops}, for any modified ascent sequence $x$
the ascent tops of $x$ together with the first element,
$x(1)=1$, form a permutation of length $\max(x)$.
Consequently $\max(x)=1+\asc(x)$ and $\Modasc_n\subseteq\Cay_n$.
To see that $|\Ascseq_n|=|\Modasc_n|$ we will give a bijection
$x\mapsto \hatx$ from $\Ascseq_n$ to $\Modasc_n$. Given an ascent
sequence $x$, let
$$M(x,j) = x', \text{ where } x'(i) = x(i) +
\begin{cases}
  1 & \text{if $i<j$ and $x(i)\geq x(j+1)$,} \\
  0 & \text{otherwise,}
\end{cases}
$$
and extend the definition of $M$ to multiple indices
$j_1$, $j_2$, \dots, $j_k$ by
$$M(x,j_1,j_2,\dots, j_k) = M\bigl(M(x, j_1,\dots, j_{k-1}), j_k\bigr).$$
We let $\hatx = M\bigl(x, \Asc(x)\bigr)$, where
$\Asc(x)=\bigl(i : x(i) < x(i+1)\bigr)$ denotes the vector of ascents of
$x$. For example, if $x=121242232$, then $\Asc(x)=(1,3,4,7)$ and
we get:
\begin{align*}
  x &= 121242232 \\
  M(x,1) &= \underline{12} 1242232\\
  M(x,1,3) &= 1\mathit{3} \underline{12} 42232 \\
  M(x,1,3,4) &= 131 \underline{24} 2232 \\
  M(x,1,3,4,7) &= 1\mathit{4}12\mathit{5}2 \underline{23} 2 = \hatx
\end{align*}

The construction described above can easily be inverted and thus the
mapping $x \mapsto \hatx$ is a bijection. 
The recursive definition of $\Modasc_n$ is equivalent to saying that
$\Modasc_n$ is the image of $\Ascseq_n$ under the $x\mapsto \hatx$
mapping. Indeed, $\Modasc_n$ was originally defined~\cite{BMCDK} in this
manner.

\subsection{Fishburn permutations}\label{Section_fishburn_perm}

\begin{figure}
$$
\fishpattern \;=\;
\begin{tikzpicture}[scale=0.40, baseline=20.5pt]
\fill[NE-lines] (1,0) rectangle (2,4);
\fill[NE-lines] (0,1) rectangle (4,2);
\draw [semithick] (0.001,0.001) grid (3.999,3.999);
\filldraw (1,2) circle (6pt);
\filldraw (2,3) circle (6pt);
\filldraw (3,1) circle (6pt);
\end{tikzpicture}
$$
\caption{Bivincular pattern $\fishpattern$ characterizing
  Fishburn permutations}\label{mesh_fishburn}
\end{figure}

Define the bivincular pattern
$\fishpattern = (231, \lbrace 1 \rbrace, \lbrace 1 \rbrace)$, as in
Figure~\ref{mesh_fishburn}. Let $\F=\Sym(\fishpattern)$. The elements of $\F$
are called \emph{Fishburn permutations}. Bousquet-Mélou~\etal~\cite{BMCDK} gave
a length-preserving bijection between ascent sequences and Fishburn
permutations. More precisely, ascent sequences encode the so called
\emph{active sites} of the Fishburn permutations. The term active site
comes from the generating tree approach to enumeration. Each vertex in
such a tree corresponds to a combinatorial object and the path from the
root to a vertex encodes the choices made in the construction of the
object. Regarding Fishburn permutations, let us construct an
element of $\F_{n+1}$ by starting from an element of $\F_n$ and inserting a
new maximum in some position.  The avoidance of the pattern
$\fishpattern$ makes some of the positions forbidden, while the others
are the active sites. More precisely, let $\pi\in\F_n$ be a
Fishburn permutation. For $i \in [n]$, if $\pi(i)=1$ let $J(i)=0$,
otherwise let $J(i)$ be the index such that $\pi(J(i))=\pi(i)-1$. 
Counting from the position to the left of the first entry of $\pi$,
position $1$ is always active and position $i+1$ is active if
and only if $J(i)<i$.
In all the other cases, the insertion of $n+1$ immediately
after $\pi(i)$ would result in an occurrence $\pi(i),n+1,\pi(J(i))$
of $\fishpattern$.
Now, the empty ascent sequence
corresponds to the empty permutation. The ascent sequence corresponding
to a nonempty Fishburn permutation $\pi \in \F$ is constructed as
follows. Start from the permutation $1$ and the sequence $1$. Record the
position in which you insert the new maximum, step by step, until you get
$\pi$. To illustrate this map consider the permutation $\pi=319764825$. It is
obtained by the following insertions, where the subscripts indicate the
labels of the active sites, while positions between consecutive elements
that have no subscript are forbidden sites.

\begin{align*}
{_1} 1{_2}
&\quad\loongmapsto{$x_2=2$}\quad {_1} 1{_2} 2{_3}\\
&\quad\loongmapsto{$x_3=1$}\quad {_1} 3\nyip 1{_2} 2{_3}\\
&\quad\loongmapsto{$x_4=2$}\quad {_1} 3\nyip 1{_2} 4{_3} 2{_4}\\
&\quad\loongmapsto{$x_5=4$}\quad {_1} 3\nyip 1{_2} 4{_3} 2{_4} 5{_5}\\
&\quad\loongmapsto{$x_6=2$}\quad {_1} 3\nyip 1{_2} 6\nyip 4{_3} 2{_4} 5{_5}\\
&\quad\loongmapsto{$x_7=2$}\quad {_1} 3\nyip 1{_2} 7\nyip 6\nyip 4{_3} 2{_4} 5{_5}\\
&\quad\loongmapsto{$x_8=3$}\quad {_1} 3\nyip 1{_2} 7\nyip 6\nyip 4{_3} 8{_4} 2{_5} 5{_6}\\
&\quad\loongmapsto{$x_8=2$}\quad {_1} 3\nyip 1{_2} 9\nyip 7\nyip 6\nyip 4{_3} 8{_4} 2{_5} 5{_6}.
\end{align*}

Therefore the ascent sequence corresponding to $\pi$ is
$x=121242232$. This procedure can also be viewed as constructing $\pi$
from a given ascent sequence by successive insertions of a new maximum
in the active site specified by the ascent sequence. Throughout this
paper we will denote this mapping from ascent sequences to Fishburn
permutations by $\fun$, so that $\fun(x)=\pi$. For a proof that
$\fun: \Ascseq\to\F$ is a bijection the interested reader is again
referred to Bousquet-Mélou~\etal~\cite{BMCDK}.

Next we recall (from~\cite{BMCDK}) the construction of a map
$\hatfun : \Modasc \to \F$ such that $\hatfun (\hatx)= \fun(x)$ for each
ascent sequence $x$. It will play a central role in transporting
patterns from ascent sequences to Fishburn permutations. As we will see,
$\hatfun$ is much easier to handle than $\fun$. The relation between the
bijections $x \mapsto \hatx$, $\fun$ and $\protect\hatfun$ is
illustrated by the commutative diagram in Figure~\ref{hat_diagram}.

\begin{figure}
$$
\begin{tikzpicture}
  \matrix (m) [matrix of math nodes,row sep=3.5em,column sep=7em,minimum width=2em]
  {
    \Ascseq & \F  \\
    \Modasc &     \\
  };
  \path[-stealth, semithick]
  (m-1-1) edge node [above, yshift=2pt] {$\fun$}
  (m-1-2)
  (m-1-1) edge node [below, xshift=-22pt, yshift=10pt]{$x \mapsto \hatx$}
  (m-2-1)
  (m-2-1) edge node [right, xshift=-3pt, yshift=-10pt] {$\hatfun$} (m-1-2);
\end{tikzpicture}
$$
\caption{How the bijections $x \mapsto \hatx$, $\fun$ and
  $\protect\hatfun$ are related}\label{hat_diagram}
\end{figure}

Let $\hatx$ be a modified ascent
sequence. Write the integers $1$ through $n$ below it, and sort the
pairs $\binom{\hatx(i)}{i}$ in ascending order with respect to the top
entry, breaking ties by sorting in descending order with respect to the
bottom entry. The resulting bottom row is the permutation
$\hatfun(\hatx)$. For example, with $\hatx=141252232$, the modified
sequence of $x=121242232$, we have
$$
\binom{\hatx}{\identity}=
\binom{1\ 4\ 1\ 2\ 5\ 2\ 2\ 3\ 2}{1\ 2\ 3\ 4\ 5\ 6\ 7\ 8\ 9}
\;\longmapsto\;
\binom{1\ 1\ 2\ 2\ 2\ 2\ 3\ 4\ 5}{3\ 1\ 9\ 7\ 6\ 4\ 8\ 2\ 5}=
\binom{\upsilon(\pi)}{\pi}
$$
To reverse this process, annotate a given Fishburn permutation
$\pi$ with its active sites as in
$\pi= {_1}31{_2}9764{_3}8{_4}2{_5}5{_6}$. Write $k$ above all entries
$\pi(j)$ that lie between active sites $k$ and $k+1$. In the example, this
forms the word $\upsilon(\pi)$ above $\pi$. Then sort the pairs
$\binom{k}{\pi(j)}$ in ascending order with respect to the bottom entry.
This defines $\hatfun^{-1}$, the inverse of the map $\hatfun$.

It turns out that it is more natural to place the identity permutation above
$\hatx$, rather than below it. Then
$$\binom{\upsilon(\pi)}{\pi}^{\!T} =\, \binom{\identity}{\hatx}$$
is a special case of transposing matrices in a sense that we describe in
the next section.

\section{The Burge transpose}\label{Section_Burge_words}

Let $\Mat_n$ be the set of matrices with nonnegative integer entries
whose every row and column has at least one nonzero entry and are such
that the sum of all entries is equal to $n$. For instance, $\Mat_2$
consists of the following five matrices:
$$
\begin{pmatrix}
  2
\end{pmatrix},\,
\begin{pmatrix}
  1 & 1
\end{pmatrix},\,
\begin{pmatrix}
  1 \\
  1
\end{pmatrix},\,
\begin{pmatrix}
  1 & 0 \\
  0 & 1
\end{pmatrix},\,
\begin{pmatrix}
  0 & 1 \\
  1 & 0
\end{pmatrix}.
$$
With each matrix $A=(a_{ij})$ in $\Mat_n$ we associate a biword in which
every column $\binom{i}{j}$ appears $a_{ij}$ times and the columns are
sorted in ascending order with respect to the top entry, breaking ties
by sorting in descending order with respect to the bottom entry. The
biwords corresponding to the five matrices above are
$$
\binom{1\,1}{1\,1},\, \binom{1\,1}{2\,1},\, \binom{1\,2}{1\,1},\,
\binom{1\,2}{1\,2},\, \binom{1\,2}{2\,1}.
$$
Note that if $i$ appears in the bottom row of such a biword, then each
$k$ such that $1\leq k<i$ also appears in the bottom row. This follows
from the requirement that each column of the corresponding matrix has at
least one nonzero entry. In other words, the bottom row is a Cayley
permutation. Similarly, the top row is a Cayley permutation. In fact, it
is a weakly increasing Cayley permutation.

Let $\WI_n$ be the subset of $\Cay_n$ consisting of the weakly
increasing Cayley permutations:
$$\WI_n = \{u\in\Cay_n: u(1)\leq u(2)\leq \dots \leq u(n)\}.
$$
To ease notation we will often write biwords as pairs. As an example,
the first two biwords in the list corresponding to matrices in $\Mat_2$
would be written $(11,11)$ and $(11, 21)$. In general, the set of
biwords corresponding to matrices in $\Mat_n$ is
$$\Bur_n = \{ (u, v)\in\WI_n\times\Cay_n: \Des(u)\subseteq \Des(v) \},
$$
where $\Des(v) = \{ i : v(i) \geq v(i+1)\}$ is the set of weak descents
of $v$. We shall call the elements of $\Bur_n$ \emph{Burge words}. This
terminology is due to Alexandersson and Uhlin~\cite{AlUh}. The
connection to Burge is with his variant of the RSK
correspondence~\cite{Bu}. Since $u$ is weakly increasing we
have $\Des(u)=\{i: u(i) = u(i+1)\}$. In particular,
$$|\Bur_n| = \sum_{v\in \Cay_n} 2^{\des(v)},
$$
where $\des(v)=|\Des(v)|$ is the number of weak descents in $v$. This is
sequence A120733 in the OEIS~\cite{Sl}.

The simple operation of transposing a matrix in $\Mat_n$ turns out to be
surprisingly useful. Assume that $A=(a_{ij})\in \Mat_n$ and that $w$ is
its corresponding biword in $\Bur_n$. Let $w^T$ denote the biword
corresponding to the transpose $A^T=(a_{ji})$ of $A$. It is easy to
compute $w^T$ without taking the detour via the matrix $A$. Turn each
column of $w$ upside down and then sort the columns as previously
described. In particular, if $\pi$ is a permutation, then
$$
\binom{\identity}{\pi}^{\!T} = \binom{\identity}{\pi^{-1}}.
$$

Also, if $\pi=\psi(\hat{x})$ is the Fishburn permutation corresponding to
the modified ascent sequence $\hat{x}$ and $\upsilon(\pi)$ is as described
in the previous section, then
$$
\binom{\upsilon(\pi)}{\pi}^{\!T} = \binom{\identity}{\hat{x}},
$$

Let $\BurF_n = \{(\upsilon(\pi), \pi): \pi\in \F_n\}$ and
$\BurMA_n = \{(\identity, \hat{x}): \hat{x}\in \Modasc_n\}$.  Then the
correspondence between Fishburn permutations and modified ascent
sequences is the identity
$$\BurF_n^T = \BurMA_n,$$
where $\BurF_n^T = \{w^T: w\in \BurF_n\}$ is the image of $\BurF_n$ under $T$.

It is clear that $\Bur_n$ is closed under transpose. In fact,
the definition of transposition given above applies to all biwords
in $\WI_n\times\Cay_n$ and it gives an alternative characterization
of the set $\Bur_n$, as seen in the following lemma.

\begin{lemma}\label{lemma_Burge_words}
  Let $w=(u,v) \in \WI_n\times\Cay_n$. Then $\Des(u)\subseteq\Des(v)$ if
  and only if $(w^T)^T=w$. Moreover, $T$ is an involution on $\Bur_n$.
\end{lemma}
\begin{proof}
  By definition of $T$, the biword $w^T$ is a Burge word. Therefore
  $(\WI_n\times\Cay_n)^T\subseteq\Bur_n$. If $\Des(u)\subseteq \Des(v)$, then
  both $w$ and $(w^T)^T$ are Burge words, and since they share the same
  set of columns, we must have $w=(w^T)^T$. Conversely, suppose that
  $(w^T)^T=w$. Then $w=z^T$, for $z=w^T$, and so $w$ is a Burge word,
  or, equivalently, $\Des(u)\subseteq \Des(v)$, as desired.  That $T$ is an
  involution on $\Bur_n$ is immediate.
\end{proof}

It is well known that the $n$th Eulerian polynomial evaluated at $2$
equals the $n$th Fubini number. That is,
\begin{equation}\label{Eulerian-of-2}
  |\Cay_n| = \sum_{\pi\in\Sym_n}2^{\des(\pi)}.
\end{equation}
The following proof is taken from Stanley~\cite[Exercise~131(a), Chapter~1]{St}.
To each pair $(\pi, E)$, with $\pi\in\Sym_n$ and $E\subseteq \Des(\pi)$,
we bijectively
associate a ballot of $[n]$: Draw a vertical bar between $\pi(i)$ and
$\pi(i+1)$ if $i$ is an ascent or $i\in E$. Thus, if $\pi=319764825$ and
$E=\{1,5\}\subseteq\Des(\pi) = \{1,3,4,5,7\}$ we get the ballot
$3|1|976|4|82|5$.

We shall reformulate this proof in terms of the transpose of Burge
words. First a definition. The \emph{direct sum} $u \oplus v$ of two
Cayley permutations $u$ and $v$ is the concatenation $uv'$, where $v'$
is obtained from $v$ by adding $\max(u)$ to each of its elements. For
instance, $12\oplus 1112\oplus 11 \oplus 1 = 123334556$. We further
extend the direct sum to sets $U$ and $V$ of Cayley
permutations:
$$U\oplus V = \{ u\oplus v: u\in U, v\in V\}.$$

Let us now return to the proof of Equation~\eqref{Eulerian-of-2}. Let $\pi$
be a permutation of $[n]$. A \emph{descending run} of $\pi$ is a maximal
sequence of consecutive descending letters
$\pi(i) > \pi(i+1) > \dots > \pi(i+d-1)$. Let $\pi=B_1 B_2 \cdots B_t$ be the
decomposition of $\pi$ into descending runs and let $\ell(i)=|B_i|$ be the
length of the $i$th descending run. The descending runs of the example
permutation $\pi=319764825$ are $31$, $9764$, $82$ and $5$. The lengths of
those runs are $2$, $4$, $2$ and $1$. The next step is to pick a weakly
increasing Cayley permutation that is a direct sum of sequences of the
same lengths as the descending runs. That is, we will pick $u$ from
$\WI_2\oplus \WI_4\oplus \WI_2\oplus \WI_1$. Since $|I_k|=2^{k-1}$ there
are $2\cdot 8\cdot 2\cdot 1 = 32$ possible choices for $u$. Say we pick
$u=12\oplus 1112\oplus 11 \oplus 1 = 123334556$. Then
$$\binom{u}{\pi}^{\!T}
= \binom{1\,2\,3\,3\,3\,4\,5\,5\,6}{3\,1\,9\,7\,6\,4\,8\,2\,5}^{\!T}
= \binom{1\,2\,3\,4\,5\,6\,7\,8\,9}{2\,5\,1\,4\,6\,3\,3\,5\,3}
= \binom{\identity}{v}
$$
and the resulting Cayley permutation is $v=251463353$, which encodes the
same ballot, $\{3\}\{1\}\{9,7,6\}\{4\}\{8,2\}\{5\}$, as in the previous example.

For $\pi\in\Sym_n$, let
$$\WI(\pi) = \WI_{\ell(1)}\oplus\cdots\oplus \WI_{\ell(t)},
$$
where $t$ is the number of descending runs of $\pi$, or, equivalently,
\begin{equation}\label{I(pi)}
  \WI(\pi) = \{ u\in \WI_n : \Des(u) \subseteq \Des(\pi) \}.
\end{equation}
Define the set
$\basis(\pi)\subseteq \Cay_n$ by
$$\bigl(\WI(\pi)\times \{\pi\}\bigr)^{\!T} = \{\identity\}\times\basis(\pi).
$$
We call $\basis(\pi)$ the \emph{Fishburn basis} of $\pi$. The reason will
become evident later. In particular,
\begin{equation}\label{card_of_basis}
  |B(\pi)| = |I(\pi)| = 2^{\des(\pi)}.
\end{equation}

Alternatively, let the underlying permutation of a ballot be obtained by
sorting elements within blocks decreasingly and then removing the curly
brackets. Thus, the underlying permutation of
$\{3\}\{1\}\{9,7,6\}\{4\}\{8,2\}\{5\}$ is $319764825$. This defines a
natural surjection from ballots to permutations and $\basis(\pi)$ is
exactly the collection of encodings of ballots whose underlying
permutation is $\pi$. In particular,
$$\bigcup_{\pi\in\Sym_n}\basis(\pi)=\Cay_n$$
in which the union is disjoint. Equation~\eqref{Eulerian-of-2} follows in view of~\eqref{card_of_basis}.

\begin{example*}
Let $\pi=3142$. The descending runs of $\pi$ are $31$ and $42$. Further,
$$\WI(3142)=\WI_2\oplus\WI_2 = \lbrace 1122,1123,1233,1234\rbrace.\qquad
$$
The Fishburn basis $\basis(3142)$ is defined by
\begin{align*}
\bigl(\WI(3142)\times \{3142\}\bigr)^{\!T} &= \{1234\}\times\basis(3142), \\
\shortintertext{where}
\bigl(\WI(3142)\times \{3142\}\bigr)^{\!T}&=
\Biggl\{
\makebox[8.5ex][l]{\ensuremath{\displaystyle\binom{1122}{3142}^{\!T}}},
\makebox[8.5ex][l]{\ensuremath{\displaystyle\binom{1123}{3142}^{\!T}}},
\makebox[8.5ex][l]{\ensuremath{\displaystyle\binom{1233}{3142}^{\!T}}},
\makebox[8.5ex][l]{\ensuremath{\displaystyle\binom{1234}{3142}^{\!T}}}
\;\Biggr\}\\
&=\Biggl\{
\makebox[8.5ex][l]{\ensuremath{\displaystyle\binom{1234}{1212}}},
\makebox[8.5ex][l]{\ensuremath{\displaystyle\binom{1234}{1312}}},
\makebox[8.5ex][l]{\ensuremath{\displaystyle\binom{1234}{2313}}},
\makebox[8.5ex][l]{\ensuremath{\displaystyle\binom{1234}{2413}}}
\;\Biggr\}.
\end{align*}
At the end we get
$\basis(3142)=\left\lbrace 1212,1312,2313,2413\right\rbrace$.  The
corresponding ballots, each of which has $3142$ as the underlying
permutation, are $\lbrace 3,1\rbrace\lbrace 4,2\rbrace$,
$\lbrace 3,1\rbrace\lbrace 4\rbrace\lbrace 2\rbrace$,
$\lbrace 3\rbrace\lbrace 1\rbrace\lbrace 4,2\rbrace$, and
$\lbrace 3\rbrace\lbrace 1\rbrace\lbrace 4\rbrace\lbrace 2\rbrace$.
\end{example*}

\section{The transport theorem}\label{Section_transport_theorem}
The main goal of this section is to prove a transport theorem for $\Cay$
and $\Sym$. We first define a notion of pattern containment on $\Bur$
for which the Burge transpose $T$ behaves nicely. More specifically, $T$
preserves pattern containment when it is used to map $\Cay$ onto $\Sym$
(Corollary~\ref{transport_lemma}).  Transporting patterns in the other
direction, that is from $\Sym$ to $\Cay$, requires the introduction of a
notion of equivalence on $\Cay$. The main result of this section,
Theorem~\ref{transport_theorem}, is a transport theorem for the
resulting equivalence classes of Cayley permutations and classical
permutations. In Section~\ref{Section_transport_F_Modasc} we will
specialize Theorem~\ref{transport_theorem} to obtain a transport theorem
for $\Modasc$ and $\F$ (Theorem~\ref{transport_theorem_modasc_fish}),
which is the most tangible application of the proposed framework. To
help the reader see where we are heading let us paraphrase this theorem here:
\textit{%
For any permutation $\sigma$ we have
$$
\F(\sigma) = \psi\bigl(\Modasc(\basis(\sigma))\bigr).
$$
In other words, the set $\F(\sigma)$ of Fishburn permutations avoiding
$\sigma$ is mapped via the bijection $\hatfun^{-1}$ to the set
$\Modasc(\basis(\sigma))$ of modified ascent sequences avoiding
all patterns in the Fishburn basis $\basis(\sigma)$.
}

Consider the map $\Gamma: \Bur_n \to \Cay_n$ defined by
$$\binom{u}{v}^{\!T} = \binom{y}{\Gamma(u,v)},$$
for any $(u,v) \in \Bur_n$. Let us write $\sort(v)$ for the word
obtained by sorting $v$ in weakly increasing order. Then
$y=\sort(v)$ and, since $T$ is an involution, $u=\sort(\Gamma(u,v))$.
For example, in the previous section we computed
$$
\binom{1\,2\,3\,3\,3\,4\,5\,5\,6}{3\,1\,9\,7\,6\,4\,8\,2\,5}^{\!T}
= \binom{1\,2\,3\,4\,5\,6\,7\,8\,9}{2\,5\,1\,4\,6\,3\,3\,5\,3},
$$
that is $\Gamma(123334556,319764825)=251463353$.

\begin{defin}\label{defin_burge_labeling}
  Let $E \subseteq \Cay$ be a set of Cayley permutations. A \emph{Burge
    labeling} on $E$ is a map $\lambda: E \to \WI$ such that
  $(\lambda(x), x)$ is a Burge word for each $x \in E$. Equivalently,
  $\Des(\lambda(x)) \subseteq \Des(x)$ for each $x \in E$.
\end{defin}

Let $\lambda$ be a Burge labeling on $E$. Then $\lambda$ induces a map
$\Gamma_\lambda: E \to \Cay$ by
$$\Gamma_\lambda (x) = \Gamma(\lambda(x),x).
$$
If $\lambda$ is injective, then $\Gamma_\lambda$ is also injective. Indeed suppose
that $\Gamma_\lambda(x)=\Gamma_\lambda(y)$. Then
$\lambda(x)=\sort(\Gamma_\lambda(x))=\sort(\Gamma_\lambda(y))=\lambda(y)$
and thus $x=y$, if $\lambda$ is injective.

This construction becomes particularly meaningful for specific
labelings. Let $\iota:\Cay \to \WI$ be defined by
$\iota(x)=\identity_n$, for each $x \in \Cay_n$. Since
$\Des(\iota(x))=\emptyset$, the mapping $\iota$ is clearly a Burge
labeling on $\Cay$.  From now on, let $\gamma=\Gamma_\iota$.

\begin{lemma}\label{sort_is_composition}
  We have $\sort(x) = x \circ \gamma(x)$ for each Cayley permutation $x$.
\end{lemma}

\begin{proof}
  To ease notation, let $\pi=\gamma(x)$. By definition of $\gamma$ we
  have $\bigl(\identity_n, x\bigr)^T\! = \bigl(\sort(x), \pi\bigr)$.
  The $i$th column of this biword is $\bigl(\sort(x)(i),\pi(i)\bigr)$.
  By definition of Burge transpose, $\bigl(\pi(i), \sort(x)(i)\bigr)$
  must be a column of $\bigl(\identity_n, x\bigr)$. Indeed, it is the
  $\pi(i)$th column of $\bigl(\identity_n, x\bigr)$, but that column is
  plainly also equal to $\bigl(\pi(i), x(\pi(i))\bigr)$, and hence
  $\sort(x)(i)=x(\pi(i))$ as claimed.
\end{proof}

\begin{remark}\label{remark_phi_gen_inv}
  If $\pi\in S_n$, then
  $\identity_n = \sort(\pi) = \pi \circ \gamma(\pi)$ by
  Lemma~\ref{sort_is_composition}.  That is, $\gamma(\pi) =
  \pi^{-1}$. In this sense, $\gamma: \Cay \to \Sym$ generalizes the
  permutation inverse to $\Cay$.
\end{remark}

\begin{remark}\label{remark_phi_on_modasc}
  Recall that, for any modified ascent sequence $x$, we have
  $$
  \bigl(\identity, x\bigr)^T =
  \bigl(\sort(x), \psi(x)\bigr),
  $$
  where $\psi: \Modasc \to \F$ is the bijection described in
  Section~\ref{Section_fishburn_perm}. Thus, restricting $\iota$ to
  $\Modasc$ gives the map $\Gamma_{\iota_{|\Modasc}} = \psi$. That is,
  $\gamma_{|_{\Modasc}} = \psi$ and in this sense $\gamma$ generalizes
  $\psi$ to $\Cay$. On the other hand, consider the map
  $\upsilon : \F \to \WI$ introduced in
  Section~\ref{Section_fishburn_perm}. It is easy to see that $\upsilon$
  is a Burge labeling on $\F$ and $\Gamma_\upsilon: \F \to \Modasc$ is
  equal to $\psi^{-1}$, the inverse map of $\psi$.
\end{remark}

Next, we extend the pattern containment relation from Cayley permutations to Burge words.

\begin{defin}
Let $(u',v')$ in $\Bur_k$ and $(u,v)$ in $\Bur_n$.
Then $(u',v')\leq (u,v)$ if there is an increasing injection
$\alpha: [k]\to [n]$ such that $u\circ\alpha$ and $v\circ\alpha$ are
order isomorphic to $u'$ and $v'$, respectively.
In other words, there is a subset of columns
determined by the indices
$\alpha([k])=\lbrace i_1,\dots,i_k\rbrace$, with $i_1<\cdots<i_k$,
such that both $u(i_1)\cdots u(i_k)\simeq u'$ and
$v(i_1)\cdots v(i_k)\simeq v'$. We also say that
$\bigl(u(i_1)\cdots u(i_k),v(i_1)\cdots v(i_k)\bigr)$ is an
\emph{occurrence} of $(u',v')$ in $(u,v)$.
\end{defin}

As an important special case, $(\identity_k,v')\leq (\identity_n,v)$
if and only if $v'\leq v$.
The next lemma shows that the Burge transpose behaves well
with respect to pattern containment on biwords.

\begin{lemma}\label{fundamental_lemma}
Let $(u,x)\in\Bur_n$ and $(v,y)\in\Bur_k$. Then:
$$
(v,y)\leq (u,x)  \iff (v,y)^T\leq(u,x)^T.
$$
\end{lemma}
\begin{proof}
  Suppose that $(v,y)\leq (u,x)$ and let
  $(u',x')=(u(i_1)\cdots u(i_k),x(i_1)\cdots x(i_k))$ be an occurrence
  of $(v,y)$ in $(u,x)$.  The relative order of any pair of columns is
  not affected by the remaining columns when $T$ is applied.  In other
  words, the effect of $T$ on $(u',x')$ is the same as the one on the
  (order isomorphic) biword $(v,y)$. Therefore $(u', v')$ is mapped by
  $T$ to an occurrence of $(v,y)^T$ in $(u,x)^T$, as desired. The
  converse follows from $T$ being an involution.
\end{proof}

\begin{corollary}\label{transport_lemma}
  Let $x\in\Cay_n$.
  \begin{enumerate}
  \item If $y\in\Cay_k$ and $y\leq x$, then $\gamma(y)\leq\gamma(x)$.
  \item If $\sigma\in\Sym_k$ and $\sigma\leq\gamma(x)$, then there exists $y\in\Cay_k$
    such that $y\leq x$ and $\gamma(y)=\sigma$.
  \end{enumerate}
\end{corollary}
\begin{proof}
  The first statement follows by letting $u=\identity_n$ and
  $v=\identity_k$ in Lemma~\ref{fundamental_lemma}.  For the second
  statement, suppose that $\sigma\leq\gamma(x)$ and let
  $\gamma(x)(i_1)\cdots\gamma(x)(i_k)$ be an occurrence of $\sigma$ in
  $\gamma(x)$. We have:
  $$
  \binom{\identity_n}{x}^{\!T} =
  \binom{\sort(x)}{\gamma(x)} =
  \begin{pmatrix}
    \dots & \sort(x)(i_1) & \dots & \sort(x)(i_k) & \dots \\
    \dots & \gamma(x)(i_1) & \dots & \gamma(x)(i_k) & \dots
  \end{pmatrix}.
  $$
  Let $v\in\WI_k$ be the only weakly increasing Cayley permutation that
  is order isomorphic to $\sort(x)(i_1)\ldots\sort(x)(i_k)$. Note that
  $(\sort(x),\gamma(x))\geq (v,\sigma)$. Thus, again by
  Lemma~\ref{fundamental_lemma}, we have:
  $$
  \binom{\identity_n}{x}=
  \binom{\sort(x)}{\gamma(x)}^{\!T}\geq
  \binom{v}{\sigma}^{\!T}=
  \binom{\identity_k}{\Gamma(v,\sigma)}.
  $$
  Therefore, for $y=\Gamma(v,\sigma)$ we have $x\ge y$ and
  $\gamma(y)=\sigma$, as desired.
\end{proof}

\begin{example*}
Let $x=251463353$. We have
$$
\binom{\identity}{x}^{\!T}=
\binom{\underline{1}\,\underline{2}\,3\,4\,5\,6\,\underline{7}\,\underline{8}\,9}
      {\underline{2}\,\underline{5}\,1\,4\,6\,3\,\underline{3}\,\underline{5}\,3}^{\!T}=
\binom{1\,\underline{2}\,3\,\underline{3}\,3\,4\,\underline{5}\,\underline{5}\,6}
      {3\,\underline{1}\,9\,\underline{7}\,6\,4\,\underline{8}\,\underline{2}\,5}=
\binom{\sort(x)}{\gamma(x)}.
$$
The underlined subset of columns highlights that $T$ maps the occurrence
$(1278,2535)$ of $(1234,1323)$ to an occurrence $(2355,1782)$
of $(1233,1342)=(1234,1323)^T$
as per Lemma~\ref{fundamental_lemma}, with $(v,y)=(1234,1323)$,
and Corollary~\ref{transport_lemma}, with $y=1323$.
On the other hand, every occurrence of $1342$ in $\gamma(x)$
does not necessarily come from an occurrence of $1323$ in $x$.  For
example, $(2356,1785)$ is the image under $T$ of $(1578,2635)$. Here
$2635\simeq 1423$ whereas in the previous case the occurrence $1782$ of
$1342$ resulted from $2535\simeq 1323$. Note that $\gamma(1423)=\gamma(1323)=1342$
as per Corollary~\ref{transport_lemma} with $\sigma=1342$.
\end{example*}

The first item of Corollary~\ref{transport_lemma} shows that when the
Burge transpose is used as a map from $\Cay$ to $\Sym$---by means of
$\gamma$---it preserves pattern containment.  The second item of the
same corollary shows that when the Burge transpose is used as a map from
$\Sym$ to $\Cay$---by means of the computation
$(\sort(x),\gamma(x))^T=(\identity_n,x)$---an occurrence of a pattern
$\sigma$ in $\gamma(x)$ is mapped to an occurrence of some $y$ in $x$,
with $\gamma(y)=\sigma$.  This leads us to define an equivalence
relation $\sim$ on $\Cay$ by $y \sim y'$ if and only if
$\gamma(y)=\gamma(y')$. Denote by $[y]$ the equivalence class
$$[y] = \lbrace y' \in \Cay: y \sim y' \rbrace
$$
of $y$, and denote by $[\Cay]$ the quotient set
$$
[\Cay] = \lbrace [y]: y \in \Cay \rbrace.
$$
Since $\sim$ is the equivalence relation induced by $\gamma$, there is a
unique injective map $\tilde\gamma$ such that the diagram
$$
\begin{tikzpicture}
  \matrix (m) [matrix of math nodes,row sep=3.5em,column sep=7em,minimum width=2em]
  {
    \Cay & \Sym  \\
    \text{$[\Cay]$} & \\
  };
  \path[-stealth, semithick]
  (m-1-1) edge node [above, xshift=-6pt, yshift=2pt] {$\gamma$}
  (m-1-2)
  (m-1-1) edge node [below, xshift=-22pt, yshift=10pt]{$x \mapsto [x]$}
  (m-2-1)
  (m-2-1) edge node [right, xshift=-8pt, yshift=-12pt] {$\tilde\gamma$} (m-1-2);
\end{tikzpicture}
$$
commutes. Furthermore, since $\gamma$ is surjective, $\tilde\gamma$ is
surjective too.  Indeed, for any permutation $\sigma$, we have
$\tilde\gamma([\sigma^{-1}])=\gamma(\sigma^{-1})=\sigma$. Thus $\tilde\gamma$ is
a bijection and the quotient set $[\Cay]$ is equinumerous with $S$, the
set of permutations. By slight abuse of notation we will write $\gamma$
for $\tilde\gamma$ as well. That is, we have two functions
$\gamma: \Cay\to \Sym$ and $\gamma: [\Cay]\to \Sym$, and it should be
clear from the context which one is referred to.
Equivalence classes of Cayley permutations up to length four are listed
in Table~\ref{table_equiv_classes}.

\begin{table}
$$
\begin{array}{ccc}
\begin{array}[t]{c|c}
\text{Equivalence class }[y] & \gamma(y)\\
\hline
\underline{1} & \underline{1}\bigskip\\
\hline
\underline{12} & \underline{12}\\
\underline{11},21 & \underline{21}\bigskip\\
\hline
\underline{123} & \underline{123}\\
\underline{122},132 & \underline{132}\\
\underline{112},213 & \underline{213}\\
212,312 & 231\\
\underline{121},231 & \underline{312}\\
\underline{111},211,221,321 & \underline{321}\bigskip\\
\hline
\underline{1234} & \underline{1234}\\
\underline{1233},1243 & \underline{1243}\\
\underline{1223},1324 & \underline{1324}\\
1323,1423 & 1342\\
\underline{1232},1342 & \underline{1423}\\
\underline{1222},1322,1332,1432 & \underline{1432}\\
\end{array}
&&
\begin{array}[t]{c|c}
\text{Equivalence class }[y] & \gamma(y)\\
\hline
\underline{1123},2134 & \underline{2134}\\
\underline{1122},1132,2133,2143 & \underline{2143}\\
2123,3124 & 2314\\
3123,4123 & 2341\\
2132,3142 & 2413\\
2122,3122,3132,4132 & 2431\\
\underline{1213},2314 & \underline{3124}\\
\underline{1212},1312,2313,2413 & \underline{3142}\\
\underline{1112},2113,2213,3214 & \underline{3214}\\
2112,3112,3213,4213 & 3241\\
2312,3412 & 3412\\
2212,3212,3312,4312 & 3421\\
\underline{1231},2341 & \underline{4123}\\
\underline{1221},1321,2331,2431 & \underline{4132}\\
\underline{1121},2131,2231,3241 & \underline{4213}\\
2121,3121,3231,4231 & 4231\\
\underline{1211},2311,2321,3421 & \underline{4312}\\
\underline{1111},2111,2211,2221,3211,3221,4321 & \underline{4321}\\
\end{array}
\end{array}
$$
\caption{Equivalence classes $[y]\in[\Cay]$ and their image $\gamma(y)$
  for $1\leq |y|\leq 4$. Modified ascent sequences and Fishburn
  permutations are underlined.}\label{table_equiv_classes}
\end{table}

Next we extend the notion of pattern containment to $[\Cay]$.

\begin{defin}
Let $[x]$ and $[y]$ in $[\Cay]$. Then $[x]\ge [y]$
if $x' \ge y'$ for some $x'\in [x]$ and $y'\in [y]$.
\end{defin}

Corollary~\ref{transport_lemma} can then be reformulated as a 
transport theorem between permutations and equivalence classes
of Cayley permutations.

\begin{theorem}[The transport theorem]\label{transport_theorem}
  Let $x,y \in \Cay$. Then
  $$
  [x] \ge [y] \iff \gamma(x) \ge \gamma(y)
  $$
  or, equivalently,
  $$
  \gamma\bigl( [\Cay][y] \bigr) = \Sym\bigl( \gamma(y) \bigr).
  $$
\end{theorem}
\begin{proof}
  Suppose that $[x]\ge[y]$, that is, $x'\ge y'$ for some $x'\in[x]$ and
  $y'\in[y]$. Then $\gamma(x')\geq\gamma(y')$ by
  Corollary~\ref{transport_lemma}. By definition of the equivalence
  relation we have $\gamma(x')=\gamma(x)$ and $\gamma(y')=\gamma(y)$,
  and hence $\gamma(x) \ge \gamma(y)$.  Conversely, if
  $\gamma(x)\geq\gamma(y)$, then by Corollary~\ref{transport_lemma}
  there exists $y'$ such that $\gamma(y')=\gamma(y)$, that is,
  $y'\in[y]$ and $x\geq y'$, from which $[x]\geq [y]$ follows.
\end{proof}

The equivalence class of $y$ is none other than the Fishburn
basis of $\sigma=\gamma(y)$:

\begin{lemma}\label{fishburn_basis}
  For $y \in \Cay$ we have $[y]=\basis (\gamma(y))$. Moreover, for each
  permutation $\sigma\in S$, we have $\basis(\sigma)=[\sigma^{-1}]$.
\end{lemma}
\begin{proof}
  Let $\sigma=\gamma(y)\in\Sym_n$. We will start by showing the inclusion
  $\basis(\sigma) \subseteq [y]$. Let $x\in\basis(\sigma)$. By definition of
  Fishburn basis we have
  \begin{align*}
  (\identity, x)^T &= (\sort(x), \sigma)
  \intertext{and $\sort(x)\in \WI(\sigma)$. On the other hand, by definition of $\gamma$ we have}
  (\identity, x)^T &= (\sort(x), \gamma(x)).
  \end{align*}
  Thus
  $\gamma(x)=\sigma=\gamma(y)$ and $x\in [y]$.
  Conversely, let $x \in [y]$; that is, $\gamma(x) = \sigma$. Then
  $$(\identity,x)^T = (\sort(x),\gamma(x)) = (\sort(x),\sigma).
  $$
  We need to show that $\sort(x) \in I(\sigma)$. Since
  $(\sort(x),\sigma)\in \Bur_n$ we have
  $\Des(\sort(x)) \subseteq \Des(\sigma)$, which in turn is equivalent to
  $\sort(x) \in I(\sigma)$ by Equation~\eqref{I(pi)}. Thus $\basis(\sigma)= [y]$ as
  claimed. Finally, $\gamma(\sigma^{-1})=\Gamma(\identity,\sigma^{-1})=\sigma$,
  and therefore $\sigma^{-1} \in \basis(\sigma)$.
\end{proof}

If $\sigma$ is a permutation, then---by Lemma~\ref{fishburn_basis}---we
can choose $\sigma^{-1}$ as representative for the Fishburn basis of
$\sigma$ and thus we have the following corollary.

\begin{corollary}\label{transport_corollary}
  If $\sigma$ is a permutation, then
  $\gamma(\Sym(\sigma)) = [\Cay][\sigma^{-1}]$.
\end{corollary}

A remarkable consequence is that the sets $\Sym(\sigma)$ and
$[\Cay][\sigma^{-1}]$ are equinumerous. In fact, we can say a bit
more. By Lemma~\ref{fishburn_basis} and Equation~\eqref{card_of_basis} we
have $|[y]|=2^{\des(\gamma(y))}$, which leads to the following result
relating the Eulerian polynomial on $\Sym_n(\sigma)$ to a polynomial
recording the distribution of (the logarithm of) sizes of equivalence
classes in $[\Cay_n][\sigma^{-1}]$. The special case $t=2$ can be seen
as a generalization of Equation~\eqref{Eulerian-of-2}.

\begin{corollary}
  For any natural number $n$ and permutation $\sigma$,
  $$
  \sum_{\pi\in \Sym_n(\sigma)} t^{\des(\pi)} \,=
  \sum_{[y]\in[\Cay_n][\sigma^{-1}]}t^{\log |[y]|},
  $$
  in which the logarithm is with respect to the base 2.
\end{corollary}

We end this section by providing some results on
the equivalence relation $\sim$ introduced in this section.
First we show that $\sim$ does not depend on our choice
of $\iota$ as Burge labeling in the definition 
of $\gamma=\Gamma_{\iota}$.

\begin{lemma}
  Let $(u,x)$ and $(u,y)$ be Burge words that have the
  same top row. If $x\sim y$, then $\Gamma(u,x) = \Gamma(u,y)$.
\end{lemma}
\begin{proof}
  Let $\pi=\gamma(x)=\gamma(y)$. Note that $\pi\in S_n$.  Let
  $u\in \WI_n$. By definition of $\Gamma$ we have
  $(u,x)^T = (\sort(x), \Gamma(u,x))$. Moreover, by
  Lemma~\ref{sort_is_composition}, $\sort(x) = x\circ \pi$ and hence
  $$
  (u, x)^T = \bigl( x\circ \pi, \Gamma(u,x)\bigr).
  $$
  It follows that $\Gamma(u,x) = u \circ \pi$. Similarly,
  $\Gamma(u,y) = u \circ \pi$, concluding the proof.
\end{proof}

Our next goal is to prove that the pattern containment
relation is a partial order on the set $[\Cay]$.

\begin{lemma}\label{equivalence_definition}
  Let $x,y \in \Cay$. The following two statements are equivalent:
  \begin{enumerate}
  \item $[x]\ge [y]$.
  \item For each $x' \in [x]$, there exists $y' \in [y]$ such that
    $x' \ge y'$.
  \end{enumerate}
\end{lemma}
\begin{proof}
  Suppose that $[x]\ge [y]$. That is, there are two Cayley permutations
  $x' \in [x]$ and $y' \in [y]$ such that $x' \ge y'$. By
  Lemma~\ref{transport_lemma}, we have $\gamma(x') \ge \gamma(y')$.  Let
  $\bar{x} \in [x]$. Then $\gamma(\bar{x})=\gamma(x') \ge
  \gamma(y')$. Thus, again by Lemma~\ref{transport_lemma}, there exists
  $\bar{y} \in [y']=[y]$ such that $\bar{x} \ge \bar{y}$, as
  desired. The other implication is trivial.
\end{proof}

\begin{prop}
  The containment relation is a partial order on $[\Cay]$.
\end{prop}
\begin{proof}
  Reflexivity is trivial. To show transitivity, suppose that
  $[x] \ge [y]$ and $[y] \ge [z]$. Then there are $x' \in [x]$ and
  $y' \in [y]$ such that $x' \ge y'$. Further, since $[y] \ge [z]$ and
  $y' \in [y]$ there is $z' \in [z]$ such that $y' \ge z'$ by
  Lemma~\ref{equivalence_definition}. Thus $x' \ge z'$ and
  $[x] \ge [z]$.  It remains to show antisymmetry. If $[x] \ge [y]$ and
  $[y] \ge [x]$, then, as in the proof of transitivity, there are three
  elements $x'\in [x]$, $y'\in [y]$ and $x''\in[x]$ such that
  $x' \ge y' \ge x''$. But then $x'=x''$, since Cayley permutations in
  the same equivalence class have the same length. Thus $x'=y'=x''$ and
  $[x]=[y]$.
\end{proof}

\section{Transport of patterns from $\F$ to $\Modasc$}\label{Section_transport_F_Modasc}

Theorem~\ref{transport_theorem} can be specialized by choosing a
representative in each equivalence class of $[\Cay]$.  Among the
resulting examples, the most significant one is that of transport of
patterns between Fishburn permutations and modified ascent
sequences. Consider the bijection $\psi=\gamma_{|_{\Modasc}}$ of
Remark~\ref{remark_phi_on_modasc}.  If $\sigma$ is a Fishburn
permutation, then $\psi^{-1}(\sigma) \in [\sigma^{-1}]$. In other words,
the map $\psi^{-1}$ picks exactly one representative in the equivalence
class $[\sigma^{-1}]$, see Table~\ref{table_equiv_classes}.  Let
$\sigma$ and $\pi$ be Fishburn permutations.
By Theorem~\ref{transport_theorem}
$$\pi \ge \sigma \;\iff\; [\psi^{-1}(\pi)] \ge [\psi^{-1}(\sigma)] = [\sigma^{-1}].
$$
By Lemma~\ref{equivalence_definition} we therefore have $\pi \ge \sigma$
if and only if $\psi^{-1}(\pi) \ge \sigma'$ for some
$\sigma' \in [\sigma^{-1}]$. Since $\psi: \Modasc\to \F$ is bijective
and $[\sigma^{-1}]=\basis(\sigma)$ is the Fishburn basis of $\sigma$ we
obtain the promised transport theorem.

\begin{theorem}[Transport of patterns from $\F$ to $\Modasc$]\label{transport_theorem_modasc_fish}
  For any permutation $\sigma$ and Cayley permutation $y$ we have
  $$
  \F(\sigma) = \gamma\bigl(\Modasc[\sigma^{-1}]\bigr)\quad\text{and}\quad
  \gamma\bigl(\Modasc[y]\bigr) = \F\bigl(\gamma(y)\bigr)
  $$
  In other words, the set $\F(\sigma)$ of Fishburn permutations avoiding
  $\sigma$ is mapped via the bijection $\hatfun^{-1}$ to the set
  $\Modasc(\basis(\sigma))$ of modified ascent sequences avoiding
  all patterns in the Fishburn basis $\basis(\sigma)$.
\end{theorem}

\begin{corollary}
  For any permutation $\sigma$ we have
  $|\F_n(\sigma)| = |\Modasc_n(\basis(\sigma))|$.
\end{corollary}

Recall that a constructive procedure for determining the Fishburn basis
$\basis(\sigma)=[\sigma^{-1}]$ was described at the end of
Section~\ref{Section_Burge_words}.
In Table~\ref{table_equiv_classes} we listed all the equivalence
classes $[y]$ and the corresponding image $\gamma(y)$ up to length four.
In each case, as stated in Theorem~\ref{transport_theorem_modasc_fish},
we have $\gamma\bigl(\Modasc[y]\bigr)=\F\bigl(\gamma(y)\bigr)$.
Equivalently, if $\sigma=\gamma(y)$ then $[y]=[\sigma^{-1}]$
and $\gamma\bigl(\Modasc[\sigma^{-1}]\bigr)=\F(\sigma)$.

Theorem~\ref{transport_theorem_modasc_fish} can be easily generalized to
Fishburn permutations avoiding a set of patterns:

\begin{theorem}\label{transport_theorem_sets_of_patterns}
Let $\Omega$ be a set of permutation patterns. Then
$$
\F(\Omega) = \gamma\bigl(\Modasc[\Omega^{-1}]\bigr),
\text{ where }
[\Omega^{-1}] = \bigcup_{\sigma \in \Omega} [\sigma^{-1}].
$$
\end{theorem}

\section{Examples}\label{Section_Examples}

Theorems~\ref{transport_theorem_modasc_fish} provides a bijection
between the set $\F(\sigma)$ of Fishburn permutations avoiding $\sigma$
and the set $\Modasc(\basis(\sigma))$ of modified ascent sequences
avoiding all the patterns in the Fishburn basis $\basis(\sigma)$.
A construction for $\basis(\sigma)$ has been explicitly described
at the end of Section~\ref{Section_Burge_words}.
In this section we collect many examples where this approach can be pushed
further by interpreting this correspondence in terms of (plain)
ascent sequences. As a corollary of the same framework we obtain, in
Section~\ref{subsection_RGF},  a transport theorem for set
partitions encoded as restricted growth functions, and the set of
permutations avoiding the vincular pattern $\dashpatt$.

Let $\Omega$ be a set of patterns and define
$$A_{\Omega}=\bigl\{ x\in\Ascseq: \hat{x}\in\Modasc[\Omega^{-1}]\bigr\}.
$$
The following diagram illustrates the bijections
described in Figure~\ref{hat_diagram} and
Theorem~\ref{transport_theorem_sets_of_patterns}.
$$
\begin{tikzpicture}
  \matrix (m) [matrix of math nodes,row sep=3.5em,column sep=7em,minimum width=2em]
  {
    A_{\Omega} & \F(\Omega)  \\
    \Modasc[\Omega^{-1}] &     \\
  };
  \path[-stealth, semithick]
  (m-1-1) edge node [above, yshift=2pt] {$\fun$}
  (m-1-2)
  (m-1-1) edge node [below, xshift=-22pt, yshift=10pt]{$x \mapsto \hatx$}
  (m-2-1)
  (m-2-1) edge node [right, xshift=-3pt, yshift=-10pt] {$\gamma=\hatfun$} (m-1-2);
\end{tikzpicture}
$$
Below we provide several cases where the set $A_{\Omega}$ can 
be described in terms of pattern avoidance, that is, where
$A_{\Omega}=\Ascseq(\cbasis(\Omega))$ for a set of Cayley
permutations $\cbasis(\Omega)$. In these cases, we obtain a transport
theorem that links ascent sequences directly to Fishburn permutations:
$$
\fun\bigl(\Ascseq(\cbasis(\Omega))\bigr)=\F(\Omega).
$$
An overview of some of the results obtained this way can be found
in Table~\ref{table_examples_transport}. In what follows we just
sketch some of the ideas and proofs, leaving the details to the reader.

\begin{table}
$$
\begin{array}{c|c|c}
\Omega & \cbasis(\Omega) & \text{Counting Sequence} \\
\hline
21 & 11 & 1,1,\dots \\
12 & 12 & 1,1,\dots \\
\hline
213 & 112 &  2^{n-1} \\
312 & 121 & 2^{n-1} \\
132 & 122 & 2^{n-1} \\
123 & 123 & 2^{n-1} \\
231 & 212 &  \text{Catalan} \\
\hline
3142 & 1212 & \text{Catalan} \\
2134 & 1123 & \text{Catalan} \\
\hline
1423 & 132 & \text{Catalan} \\
3412 & 312 & A202062 \\
\hline
231,4132 & 212,221 & \text{Odd indexed Fibonacci}\\
231,4123 & 212,231 & A116703 \\
231,4312 & 212,211,321 & \text{Odd indexed Fibonacci}\\
\end{array}
$$
\caption{Sets of patterns $\Omega$ and $\cbasis(\Omega)$ such that
$\F(\Omega)=\fun\bigl(\Ascseq(\cbasis(\Omega))\bigr)$.}\label{table_examples_transport}
\end{table}

\subsection{Transport of a single pattern}\label{subs_single_pattern}
 
When $\Omega=\lbrace \sigma \rbrace$ is a singleton it is
sometimes possible to show that there is a pattern $y\in\Cay$
such that $A_{\lbrace \sigma\rbrace}=\Ascseq(y)$.
Equivalently, $x\in\Ascseq(y)$ if and only if $\hat{x}\in\Modasc[\sigma^{-1}]$.
As a result we get the transport of a single pattern:
$$
\fun\bigl(\Ascseq(y)\bigr)=\F(\sigma).
$$
Below we show the case $\sigma=2134$, the other cases being similar.

\begin{example*}
Let $\sigma=2134$. Observe that $[2134^{-1}]=\lbrace 1123,2134\rbrace$.
We shall prove that $A_{\lbrace 2134\rbrace}=\Ascseq(1123)$,
that is, $x\in\Ascseq(1123)$ if and only if
$\hat{x}\in\Modasc[2134^{-1}]=\Modasc(1123,2134)$.
We first prove that $\Modasc(1123,2134)=\Modasc(1123)$.

\begin{lemma}
We have $\Modasc(1123,2134)=\Modasc(1123)$.
\end{lemma}
\begin{proof}
  The inclusion $\Modasc(1123,2134)\subseteq\Modasc(1123)$ is trivial.
  To prove the opposite inclusion, we proceed by contraposition. Suppose
  that $\hat{x}\ge 2134$ and let
  $\hat{x}(i_1)\hat{x}(i_2)\hat{x}(i_3)\hat{x}(i_4)$ be an occurrence of
  $2134$ in $\hat{x}$. Let $j$ be the index of the leftmost occurrence
  of the integer $\hat{x}(i_2)$ in $\hat{x}$.  If $j<i_2$, then
  $\hat{x}(j)\hat{x}(i_2)\hat{x}(i_3)\hat{x}(i_4)$ is an occurrence of
  $1123$. Otherwise, if $j=i_2$, then
  $\hat{x}(i_2)$ is an ascent top by Lemma~\ref{lemma_ascent_tops}, that
  is, $\hat{x}(i_2-1)<\hat{x}(i_2)$.  Therefore we can repeat the same
  argument on the occurrence
  $\hat{x}(i_1)\hat{x}(i_2-1)\hat{x}(i_3)\hat{x}(i_4)$ of $2134$ until
  we eventually find an occurrence of $1123$.
\end{proof}

\begin{lemma}
Let $x\in\Ascseq$. Then $x\ge 1123$ if and only if $\hat{x}\ge 1123$.
\end{lemma}
\begin{proof}
We proceed by induction on the length $n$ of $x$. If $n\le 3$
there is nothing to prove. Otherwise, consider the following
two cases:
\begin{itemize}
\item If $x$ contains exactly one occurrence of the integer $1$, then
$x=1y^{+1}$, where $y^{+1}$ is obtained by increasing by one each
entry of some ascent sequence $y$. Similarly, $\hat{x}=1\hat{y}^{+1}$.
By the inductive hypothesis, $y\ge 1123$ if and only if $\hat{y}\ge 1123$.
Thus the same holds for $y^{+1}$ and $\hat{y}^{+1}$, and also for
$x$ and $\hat{x}$, as desired.
\item Suppose that $x$ contains at least two occurrences of the
integer $1$. Let $i$ be the index of the rightmost occurrence of $1$.
Note that $x(i-1)\le x(i)$, and hence the sequence
$y=x(1)\cdots x(i-1)x(i+1)\cdots x(n)$ obtained from $x$ by removing
$x(i)$ is an ascent sequence. In particular, $\hat{y}$ is obtained
by removing the entry $\hat{x}(i)$ from $\hat{x}$. We conclude by
applying the inductive hypothesis to $y$ and $\hat{y}$.\qedhere
\end{itemize}
\end{proof}

As a corollary of the previous two results, we have
$A_{\lbrace 2134\rbrace}=\Ascseq(1123)$ and hence
$$
\fun\bigl(\Ascseq(1123)\bigr)=\F(2134).
$$
In 2011 Duncan and Steingr\'imsson~\cite{DS} conjectured that
$|\Ascseq_n(1123)|=C_n$, the $n$th Catalan number. Three years later
this conjecture was settled in the affirmative by Mansour and
Shattuck~\cite{MS14}. Gil and Weiner~\cite{GW} independently proved that
$|\F_n(2134)|=C_n$ in 2019.
\end{example*}

\begin{example*}
The sets $\F(1423)$ and $\F(3412)$ can be dealt 
with similarly.
For instance, we have $A_{\lbrace 1423\rbrace}=\Ascseq(132)$
and hence
$$
\fun\bigl(\Ascseq(132)\bigr)=\F(1423).
$$
The set $\Ascseq(132)$ has been enumerated in \cite{DS}, whereas
the enumeration of $\F(1423)$ can be found in \cite{GW}. Both
sets are enumerated by the Catalan numbers.
Analogously, $A_{\lbrace 3412\rbrace}=\Ascseq(312)$, which leads to
$$
\fun\bigl(\Ascseq(312)\bigr)=\F(3412).
$$
No formula is currently known for the enumeration of $\Ascseq(312)$ or
$\F(3412)$. It has been conjectured~\cite{CCF} that $312$-sortable
permutations share the same enumeration.
\end{example*}

\subsection{Transport of sets of patterns}

Baxter and Pudwell~\cite{BP} showed that~$\Ascseq(212,221)$ is enumerated
by the odd indexed Fibonacci numbers (sequence A001519 in the OEIS~\cite{Sl}) by
providing a generating tree for this set. They also showed
that~$\Ascseq(212,231)$ and~$\Sym(231,4123)$ admit an isomorphic
generating tree (via the transfer matrix method and using Maple).
An alternative proof of both results can be obtained as a simple
corollary of our framework. Indeed,
$$
\fun\bigl(\Ascseq(212,221)\bigr)=\Sym(231,4132),
$$
which is well known to be enumerated by the odd indexed Fibonacci numbers, and
$$
\fun\bigl(\Ascseq(212,231)\bigr)=\Sym(231,4123).
$$
Note that $\F(231,4132)=\Sym(231,4132)$ and
$\F(231,4123)=\Sym(231,4123)$, since $231$ is the underlying permutation
of $\fishpattern$. Therefore, as outlined in the previous examples
(see also Figure~\ref{diagram_pairs} for the pair $(231,4123)$), to
prove both results it is sufficient to show that
$A_{\lbrace 231,4132\rbrace}=\Ascseq(212,221)$ and
$A_{\lbrace 231,4123\rbrace}=\Ascseq(212,231)$, respectively.
This can be achieved with a simple case by case analysis,
whose details are omitted.

\begin{figure}
$$
\begin{tikzpicture}
  \matrix (m) [matrix of math nodes,row sep=3.5em,column sep=7em,minimum width=2em]
  {
    A_{\lbrace 231,4123\rbrace} & \F(231,4123)  \\
    \Modasc([231^{-1}],[4123^{-1}]) & \\
  };
  \path[-stealth, semithick]
  (m-1-1) edge node [above, xshift=-6pt, yshift=2pt] {$\fun$}
  (m-1-2)
  (m-1-1) edge node [below, xshift=-22pt, yshift=10pt]{$x \mapsto \hat{x}$}
  (m-2-1)
  (m-2-1) edge node [right, xshift=-8pt, yshift=-12pt] {$\gamma$} (m-1-2);
\end{tikzpicture}
$$
\caption{The diagram illustrating the transport of patterns between
$\F(231,4123)=\Sym(231,4123)$ and $A_{231,4123}=\Ascseq(212,231)$.
}\label{diagram_pairs}
\end{figure}

In a similar fashion, an example of transport of a triple of
patterns is the following:
$$
\fun\bigl(\Ascseq(212,211,321)\bigr) = \Sym(231,4312),
$$
where it is easy to derive that $|\Sym_n(231,4312)|$ is an odd indexed
Fibonacci number.

\subsection{$\RGF$ as set of representatives for $[\Cay]$}\label{subsection_RGF}

In Section~\ref{Section_transport_F_Modasc} we specialized
Theorem~\ref{transport_theorem} by choosing $\Modasc$ as set of
representatives for the equivalence classes $[x]\in[\Cay]$ such
that $\gamma(x)\in\F$. The same approach can be replicated
in order to obtain a transport theorem for restricted growth
functions and the set of permutations avoiding the vincular pattern
$\dashpatt=(231,\lbrace 1\rbrace,\lbrace\rbrace)$, which is depicted
in Figure~\ref{bivinc_patt_RGF}. Note that every permutation avoiding
$\dashpatt$ is a Fishburn permutation; in symbols,
$S(\dashpatt)\subseteq\F$.
Below we outline the main ideas and give a couple of examples,
leaving a deeper investigation of this notion of transport for
future work.

A Cayley permutation $x\in\Cay_n$ is a \emph{restricted growth function}
if $x(1)=1$ and $x(i+1)\le \max\lbrace x(1),\dots,x(i)\rbrace+1$,
for each $i=1,2,\dots,n-1$. Let $\RGF_n$ be the set of restricted
growth functions of length $n$ and let $\RGF=\cup_{n\geq 0}\,\RGF_n$.
It is well known that restricted growth
functions encode set partitions in the same manner as Cayley
permutations encode ballots. Therefore the cardinality of $\RGF_n$
is equal to the $n$th Bell number (sequence A000110 in~\cite{Sl}).
Pattern avoidance on $\RGF$ has been extensively studied~\cite{CDDGGPS,JM,Sa}.

Recall from Section~\ref{Section_Burge_words}
that $\basis(\pi)$ is the set of Cayley permutations that
encode ballots whose underlying permutation is $\pi$.
More precisely, let $x\in\basis(\pi)$ and let
$P_x$ be the ballot encoded by $x$, where as usual
$x(i)=j$ if $i\in B_j$. Then $\gamma(x)=\pi$ is the permutation
obtained from $P_x$ by sorting blocks decreasingly and removing
the curly brackets.
For instance, $x=112132341$ is the Cayley permutation that encodes
the ballot $P_x=\lbrace 9,4,2,1\rbrace\lbrace 6,3\rbrace
\lbrace 7,5\rbrace\lbrace 8\rbrace$ and we have
$$
\binom{\identity}{x}^{\!T}=
\binom{1\ 2\ 3\ 4\ 5\ 6\ 7\ 8\ 9}{1\ 1\ 2\ 1\ 3\ 2\ 3\ 4\ 1}^{\!T}=
\binom{1\ 1\ 1\ 1\ 2\ 2\ 3\ 3\ 4}{9\ 4\ 2\ 1\ 6\ 3\ 7\ 5\ 8}=
\binom{\sort(x)}{\gamma(x)}.
$$
Note that $\gamma(x)=942163758$ is the permutation underlying
$P_x$ and $\sort(x)$ is a Burge labeling of $\gamma(x)$.


\begin{figure}
$$
\dashpatt \;=\;
\begin{tikzpicture}[scale=0.4, baseline=20pt]
\fill[NE-lines] (1,0) rectangle (2,4);
\draw [semithick] (0,1) -- (4,1);
\draw [semithick] (0,2) -- (4,2);
\draw [semithick] (0,3) -- (4,3);
\draw [semithick] (1,0) -- (1,4);
\draw [semithick] (2,0) -- (2,4);
\draw [semithick] (3,0) -- (3,4);
\filldraw (1,2) circle (6pt);
\filldraw (2,3) circle (6pt);
\filldraw (3,1) circle (6pt);
\end{tikzpicture}
$$
\caption{Vincular pattern such that $\gamma(\RGF)=\Sym(\dashpatt)$.}\label{bivinc_patt_RGF}
\end{figure}

Now, let $\pi$ be a permutation. Let $\pi=B_1B_2\cdots B_k$
be the decomposition of $\pi$ into descending runs and let
$\ell(i)=|B_i|$. Denote by $1^{\ell(i)}=1\dots 1$ the
concatenation of $\ell(i)$ copies of $1$.
We call $\pi\mapsto 1^{\ell(1)}\oplus\cdots\oplus 1^{\ell(k)}$ the
\emph{minimal Burge labeling} of $\pi$ because it is using the least
number of distinct integers. For an alternative way of viewing this,
observe that $\max(\lambda(\pi))\geq \asc(\pi)+1$ for any Burge labeling
$\lambda:\Sym\to\WI$ and that the minimal Burge labeling is the unique
labeling for which equality is obtained.
An overview of the Burge labelings used in this paper can be found in
Table~\ref{table_labelings}. Continuing the 
previous example (with $x=112132341$) we have
$x\in\RGF$ and $\gamma(x)\in\Sym(\dashpatt)$. The decomposition of
$\gamma(x)$ into decreasing runs is $\gamma(x)=9421|63|75|8$, where
minima of blocks are in increasing order and
$\sort(x)=1111\oplus 11 \oplus 11 \oplus 1 = 111122334$ is the minimal
Burge labeling of $\gamma(x)$.
The following proposition, whose easy proof is omitted,
characterizes the cases where a ballot is encoded by a
restricted growth function.

\begin{prop}\label{prop_RGF_equiv}
Let $x\in\Cay$ and let $(u,\pi)=(\identity,x)^{\!T}$.
Then $x\in\RGF$ if and only if $\pi\in\Sym(\dashpatt)$
and $u$ is the minimal Burge labeling of $\pi$.
\end{prop}

\begin{table}
$$
\begin{array}{c|c|c}
\text{Set}\ E & \text{Labeling}\ \lambda & \Gamma_\lambda(E) \\
\hline
\Sym & \text{Maximal labeling}\ (\iota) & \Sym^{-1} \\
\F & \upsilon & \Modasc \\
\Sym & \tilde{\upsilon} & \xset \\
\Sym(\dashpatt) & \text{Minimal labeling} & \RGF \\
\end{array}
$$
\caption{Burge labelings $\lambda:E\to\WI$ on $E\subseteq\Sym$
and the resulting sets $\Gamma_{\lambda}(E)\subseteq\Cay$.}\label{table_labelings}
\end{table}

An immediate corollary of Proposition~\ref{prop_RGF_equiv} is
that the map $\gamma$ is injective on $\RGF$ and
$$
\gamma(\RGF)=\Sym(\dashpatt).
$$
This gives an alternative perspective on the original proof~\cite{Cl}
that $\Sym(\dashpatt)$ is enumerated by the Bell numbers.
The operation that associates the permutation $\pi=\gamma(x)$
to the set partition $x\in\RGF$ is similar to the \emph{flattening}
of set partitions considered by Callan~\cite{Cal}, with the difference
that blocks are sorted increasingly in their case.
According to our notion, a permutation is the flattening of
a set partition (through $\gamma$) if and only if it avoids $\dashpatt$.

From now on, let $\Flat=\Sym(\dashpatt)$. Since $\gamma$ is
injective on $\RGF$, we can specialize Theorem~\ref{transport_theorem}
by choosing $\RGF$ as a set of representatives for $[\Cay]$ to
get a transport theorem for $\RGF$ and $\Flat$:

\begin{theorem}[Transport of patterns from $\Flat$ to $\RGF$]\label{transport_theorem_RGF}
  For any permutation $\sigma$ and Cayley permutation $y$ we have
  $$
  \Flat(\sigma) = \gamma\bigl(\RGF[\sigma^{-1}]\bigr)\quad\text{and}\quad
  \gamma\bigl(\RGF[y]\bigr) = \Flat\bigl(\gamma(y)\bigr)
  $$
  In other words, the set $\RGF[\sigma^{-1}]$ of restricted growth functions
  avoiding all the patterns in $[\sigma^{-1}]$ is mapped via the bijection
  $\gamma$ to the set $\Flat(\sigma)$ of flattened partitions
  avoiding $\sigma$. In particular, we have
  $|\Flat_n(\sigma)| = |\RGF_n[\sigma^{-1}]|$.
\end{theorem}

We end this section by providing two instances of transport
between $\RGF$ and $\Flat$.

\begin{example*}
By Theorem~\ref{transport_theorem_RGF} we have
$$\Flat(2341) = \gamma\bigl(\RGF[2341^{-1}]\bigr) =
\gamma\bigl(\RGF(4123,3123)\bigr).
$$
It is easy to show that $\RGF(4123,3123)=\RGF(3123)$
and $\Flat(2341)=\F(2341)$. Thus,
$$\gamma\bigl(\RGF(3123)\bigr)=\F(2341).
$$
Note that $\RGF(3123)=\RGF(123123)$. Restricted growth
functions avoiding $123123$ are sometimes called $3$-noncrossing
set partitions~\cite{JM}. The enumeration of both
$\RGF(123123)$ and $\F(2341)$ is currently unknown.
\end{example*}

\begin{example*}
By Theorem~\ref{transport_theorem_RGF} we have
$$
\Flat(1342) = \gamma\bigl(\RGF[1342^{-1}]\bigr) =
\gamma\bigl(\RGF(1423,1323)\bigr).
$$
Moreover, $\RGF(1423,1323)=\RGF(1323)$ and
$\Flat(1342)=\F(1342)$. Thus,
$$
\gamma\bigl(\RGF(1323)\bigr)=\F(1342).
$$
The set $\F(1342)$ is enumerated by the binomial transform of
the Catalan numbers~\cite{GW}.
\end{example*}


\section{Picking a representative for each equivalence class}\label{Section_lift_of_psi}

In Theorem~\ref{transport_theorem_modasc_fish} we exploited the map
$\psi: \Modasc \to \F$ and its inverse $\psi^{-1}$ to transport patterns
from $\F$ to $\Modasc$. It seems natural to push this approach further
by ``lifting'' $\psi^{-1}$ to a map whose domain is $\Sym$, the set of
all permutations, thus extending the reach of
Theorem~\ref{transport_theorem_modasc_fish}. In effect, we will define a
map, called $\eta$, that picks a representative for each equivalence
class in $[\Cay]$. The set of representatives will be called $\xset$ and
the lifted map will be $\eta: \Sym\to \xset$. We now detail this
construction.

Remark~\ref{remark_phi_on_modasc} shows that $\psi^{-1}=\Gamma_\upsilon$
is the map induced by the Burge labeling $\upsilon$ of Fishburn
permutations described in Section~\ref{Section_fishburn_perm}. Let $\pi$
be a Fishburn permutation. Recall that $\upsilon(\pi)$ is obtained by
\begin{enumerate}
\item annotating $\pi$ with its active sites with respect to the Fishburn
  pattern $\fishpattern$;
\item writing $k$ above all entries $\pi(j)$ that lie between active sites
  $k$ and $k+1$.
\end{enumerate}
Due to the avoidance of $\fishpattern$, the site between $\pi(i)$
and $\pi(i+1)$ is active if and only if $J(i)<i$, where
$\pi(J(i))=\pi(i)-1$. In addition, the sites before $\pi(1)$ and
after $\pi(j)=1$ are always considered active. From now on, we call
these sites \emph{$\fishpattern$-active}.

We wish to lift the map $\psi^{-1}: \F \to \Modasc$ to a map $\eta$,
with $\Sym$ as its domain, by extending the labeling $\upsilon$ to a
labeling $\tilde{\upsilon}$ on $\Sym$. The lifted map $\eta$ will then
be $\eta=\Gamma_{\tilde{\upsilon}}$.

Let $\pi$ be a permutation. We stipulate that the site between $\pi(i)$ and $\pi(i+1)$
is \emph{$\eta$-active} if $J(i)<i$ or $\pi(i)<\pi(i+1)$.  In addition,
the sites before $\pi(1)$ and after $\pi(j)=1$ are always considered
$\eta$-active.  The labeling $\tilde{\upsilon}:\Sym_n \to \WI_n$ is
defined by
\begin{enumerate}
\item annotating $\pi$ with its $\eta$-active sites;
\item writing $k$ above all entries $\pi(j)$ that lie between
  $\eta$-active sites $k$ and $k+1$.
\end{enumerate}

Now, $\tilde{\upsilon}$ is a Burge labeling
on $\Sym$. Indeed the site between $\pi(i)$ and $\pi(i+1)$ is $\eta$-active
if $\pi(i)<\pi(i+1)$, therefore
$\Des(\tilde{\upsilon}(\pi)) \subseteq \Des(\pi)$. Next we prove that
$\tilde{\upsilon}=\upsilon$ on Fishburn permutations.

\begin{lemma}\label{lemma_lift_upsilon}
  If $\pi$ is a Fishburn permutation, then
  $\upsilon(\pi)=\tilde{\upsilon}(\pi)$.
\end{lemma}
\begin{proof}
  We will show that each site of a Fishburn permutation $\pi$ is
  $\fishpattern$-active if and only if it is $\eta$-active. The sites
  before $\pi(1)$ and after $\pi(n)$ are both $\fishpattern$-active and
  $\eta$-active by definition. Consider the site between $\pi(i)$ and
  $\pi(i+1)$, for $1 \le i <n$. If the site is $\fishpattern$-active,
  then $J(i)<i$ and thus it is also $\eta$-active. Conversely, suppose
  that the site is $\eta$-active. If $J(i)<i$, then it is also
  $\fishpattern$-active. Otherwise, if $J(i)>i$, we must have
  $\pi(i)<\pi(i+1)$. But then $\pi(i)\pi(i+1)\pi(J(i))$ is an occurrence of
  $\fishpattern$ in $\pi$, which is impossible.
\end{proof}

Since $\tilde{\upsilon}$ is a Burge labeling of $\Sym$ and the
restriction of $\tilde{\upsilon}$ to $\F$ coincides with $\upsilon$, the
map $\eta=\Gamma_{\tilde{\upsilon}}: \Sym \to \Cay$ lifts the map
$\Gamma_{\upsilon}: \F \to \Modasc$. Moreover, since
$\tilde{\upsilon}$ is injective, $\eta$ is also injective, as shown
below Definition~\ref{defin_burge_labeling}. In other words, $\eta$
picks one representative in the equivalence class $[\pi^{-1}]$, for each
permutation $\pi$. If $\pi$ is a Fishburn permutation, $\eta$ chooses the same
element as $\psi^{-1}$. Let
$$\xset = \eta(\Sym).
$$
Note that $\Modasc \subseteq \xset$ by
Lemma~\ref{lemma_lift_upsilon}. In Table~\ref{table_sym_xset} we list permutations
and members of $\xset$ of length one through four. We also indicate which
ones are Fishburn permutations and modified ascent sequences,
respectively.

\begin{table}
  $$
    \begin{array}{ccc}
      \begin{array}[t]{l}
        \\[-14pt]
        \begin{array}{c|c|c}
          \pi\in\Sym & \eta(\pi)\in\xset & \eta(\pi)\in\Modasc\,? \\
          \hline
          1 & 1 & \checkmark
        \end{array}
        \\ \\
        \begin{array}{c|c|c}
          \pi\in\Sym & \eta(\pi)\in\xset & \eta(\pi)\in\Modasc\,? \\
          \hline
          12 & 12 & \checkmark \\
          21 & 11 & \checkmark
        \end{array}
        \\ \\
        \begin{array}{c|c|c}
          \pi\in\Sym & \eta(\pi)\in\xset & \eta(\pi)\in\Modasc\,? \\
          \hline
          123 & 123 & \checkmark \\
          132 & 122 & \checkmark \\
          213 & 112 & \checkmark \\
          231 & 312 & \\
          312 & 121 & \checkmark \\
          321 & 111 & \checkmark
        \end{array}
        \\ \\
        \begin{array}[t]{c|c|c}
          \pi\in\Sym & \eta(\pi)\in\xset & \eta(\pi)\in\Modasc\,? \\
          \hline
          1234 & 1234 & \checkmark \\
          1243 & 1233 & \checkmark \\
          1324 & 1223 & \checkmark \\
          1342 & 1423 & \\
        \end{array}\\
      \end{array}&&
      \begin{array}[t]{l}
      \begin{array}[t]{c|c|c}
        \pi\in\Sym & \eta(\pi)\in\xset & \eta(\pi)\in\Modasc\,? \\
        \hline
        1423 & 1232 & \checkmark \\
        1432 & 1222 & \checkmark \\
        2134 & 1123 & \checkmark \\
        2143 & 1122 & \checkmark \\
        2314 & 3124 & \\
        2341 & 4123 & \\
        2413 & 2132 & \\
        2431 & 3122 & \\
        3124 & 1213 & \checkmark  \\
        3142 & 1312 & \checkmark \\
        3214 & 1112 & \checkmark \\
        3241 & 3112 & \\
        3412 & 3412 & \\
        3421 & 3312 & \\
        4123 & 1231 & \checkmark \\
        4132 & 1221 & \checkmark \\
        4213 & 1121 & \checkmark \\
        4231 & 3121 & \\
        4312 & 1211 & \checkmark \\
        4321 & 1111 & \checkmark \\
      \end{array}
      \end{array}
    \end{array}
  $$
  \caption{Permutations and corresponding members of $\xset$}\label{table_sym_xset}
\end{table}

We defined $\eta$ as the function
$\Gamma_{\tilde{\upsilon}}:\Sym \to \Cay$. We then defined its range to
be $X$. From now on we will consider $\eta$ as a function
$\eta: \Sym \to \xset$. It is clearly a bijection\footnote{Is there some
  easier way (than using this bijection) to see that $|\xset_n|=n!$?} and we
have the following transport theorem.

\begin{theorem}[Transport of patterns from $\Sym$ to $\xset$]\label{transport_theorem_xset_sym}
  For any permutation $\sigma$ and Cayley permutation $y$ we have
  $$
    \Sym(\sigma) = \gamma\bigl(\xset[\sigma^{-1}])\bigr)
    \quad\text{and}\quad
    \gamma\bigl(\xset[y]\bigr) = \Sym\bigl(\gamma(y)\bigr).
  $$
  In other words, $\Sym(\sigma)$ is mapped via the bijection $\eta$ to
  $\xset(\basis(\sigma))$.
\end{theorem}

As a direct consequence, $\Sym(\sigma)$ and $\xset(\basis(\sigma))$ are
equinumerous subsets of $\Cay$.

\begin{example*}
  For each natural number $n$, we have
  \begin{align*}
    |\Sym_n(1324)| &= |X_n(1223, 1324)|; \\
    |\Sym_n(4231)| &= |X_n(2121, 3121, 3231, 4231)|.
  \end{align*}
\end{example*}

The rest of this section is devoted to
describing the set $\xset$. Mesh patterns on
Cayley permutations were recently introduced by Cerbai~\cite{Ce}. They
are defined like mesh patterns on permutations, but with additional
regions to account for the possibility of having repeated
elements. Instead of giving a formal definition, we refer the reader to
\cite{Ce} and Figure~\ref{mesh_patterns_X_modasc}. From now on,
let $\mathfrak{a}$, $\mathfrak{b}$, $\mathfrak{c}$ and $\mathfrak{d}$ be
the mesh patterns depicted in Figure~\ref{mesh_patterns_X_modasc}.

\begin{figure}
$$
\mathfrak{a} \,=\,
\begin{tikzpicture}[scale=0.50, baseline=19pt]
\fill[NE-lines] (2.15,0) rectangle (2.85,3);
\draw [semithick] (0,0.85) -- (4,0.85);
\draw [semithick] (0,1.15) -- (4,1.15);
\draw [semithick] (0,1.85) -- (4,1.85);
\draw [semithick] (0,2.15) -- (4,2.15);
\draw [semithick] (0.85,0) -- (0.85,3);
\draw [semithick] (1.15,0) -- (1.15,3);
\draw [semithick] (1.85,0) -- (1.85,3);
\draw [semithick] (2.15,0) -- (2.15,3);
\draw [semithick] (2.85,0) -- (2.85,3);
\draw [semithick] (3.15,0) -- (3.15,3);
\filldraw (1,2) circle (5pt);
\filldraw (2,1) circle (5pt);
\filldraw (3,2) circle (5pt);
\end{tikzpicture}
\qquad
\mathfrak{b} \,=\,
\begin{tikzpicture}[scale=0.50, baseline=19pt]
\fill[NE-lines] (1.15,0) rectangle (1.85,3);
\fill[NE-lines] (0,0.85) rectangle (0.85,1.15);
\draw [semithick] (0,0.85) -- (3,0.85);
\draw [semithick] (0,1.15) -- (3,1.15);
\draw [semithick] (0,1.85) -- (3,1.85);
\draw [semithick] (0,2.15) -- (3,2.15);
\draw [semithick] (0.85,0) -- (0.85,3);
\draw [semithick] (1.15,0) -- (1.15,3);
\draw [semithick] (1.85,0) -- (1.85,3);
\draw [semithick] (2.15,0) -- (2.15,3);
\filldraw (1,2) circle (5pt);
\filldraw (2,1) circle (5pt);
\end{tikzpicture}
\qquad
\mathfrak{c} \,=\,
\begin{tikzpicture}[scale=0.50, baseline=26pt]
\fill[NE-lines] (0.15,0.85) rectangle (0.85,1.85);
\fill[NE-lines] (1.15,0.85) rectangle (1.85,1.85);
\fill[NE-lines] (2.15,0) rectangle (2.85,4);
\fill[NE-lines] (3.15,1.15) rectangle (4,2.15);
\draw [semithick] (0,0.85) -- (4,0.85);
\draw [semithick] (0,1.15) -- (4,1.15);
\draw [semithick] (0,1.85) -- (4,1.85);
\draw [semithick] (0,2.15) -- (4,2.15);
\draw [semithick] (0,2.85) -- (4,2.85);
\draw [semithick] (0,3.15) -- (4,3.15);
\draw [semithick] (0.85,0) -- (0.85,4);
\draw [semithick] (1.15,0) -- (1.15,4);
\draw [semithick] (1.85,0) -- (1.85,4);
\draw [semithick] (2.15,0) -- (2.15,4);
\draw [semithick] (2.85,0) -- (2.85,4);
\draw [semithick] (3.15,0) -- (3.15,4);
\filldraw (1,2) circle (5pt);
\filldraw (2,3) circle (5pt);
\filldraw (3,1) circle (5pt);
\end{tikzpicture}
\qquad
\mathfrak{d} \,=\,
\begin{tikzpicture}[scale=0.50, baseline=19pt]
\fill[NE-lines] (1.15,0) rectangle (1.85,3);
\fill[NE-lines] (0,0.85) rectangle (0.85,1.85);
\fill[NE-lines] (2.15,1.15) rectangle (3,2.15);
\draw [semithick] (0,0.85) -- (3,0.85);
\draw [semithick] (0,1.15) -- (3,1.15);
\draw [semithick] (0,1.85) -- (3,1.85);
\draw [semithick] (0,2.15) -- (3,2.15);
\draw [semithick] (0.85,0) -- (0.85,3);
\draw [semithick] (1.15,0) -- (1.15,3);
\draw [semithick] (1.85,0) -- (1.85,3);
\draw [semithick] (2.15,0) -- (2.15,3);
\filldraw (1,2) circle (5pt);
\filldraw (2,1) circle (5pt);
\end{tikzpicture}
$$
\caption{Mesh patterns such that
  $\Modasc=\Cay(\mathfrak{a},\mathfrak{b})$ and
  $\xset=\Cay(\mathfrak{a},\mathfrak{c},\mathfrak{d})$}\label{mesh_patterns_X_modasc}
\end{figure}

Recall from Lemma~\ref{lemma_ascent_tops} that the
ascent tops of a modified ascent sequence $x$ together with the first
element, $x(1)=1$, form a permutation of length $\max(x)$. The converse
is also true. To be precise, let
$$\asctops(x)= \{(1,1)\}\cup \{(i, x(i)): 1 < i \le n,\, x(i-1) < x(i)\}
$$ be
the set of ascent tops and their indices---including the first
element---and let
$$\nub(x) = \{(\min x^{-1}(j), j): 1\leq j\leq \max(x) \}
$$ be the set of
first occurrences and their indices. Then
$$
\Modasc = \{x\in\Cay: \asctops(x) = \nub(x) \},
$$
giving an arguably simpler characterization of modified ascent
sequences than the one given in Section~\ref{Section_asc_seq}.
This can be equivalently expressed in terms of avoidance of
the two mesh patterns $\mathfrak{a}$ and $\mathfrak{b}$.

\begin{theorem}\label{modasc_mesh_patterns}
  We have $\Modasc = \Cay(\mathfrak{a},\mathfrak{b})$, and hence the two
  sets $\lbrace \mathfrak{a},\mathfrak{b} \rbrace$ and
  $\lbrace 11,\fishpattern \rbrace$ are Wilf-equivalent.
\end{theorem}

\begin{proof}
  Let $x\in\Cay_n$ be a Cayley permutation. We start by showing
  that if $x$ contains $\mathfrak{a}$ or $\mathfrak{b}$, then $x$ is not
  a modified ascent sequence. Suppose that $x(i) x(j) x(j+1)$ is an
  occurrence of $\mathfrak{a}$ in $x$. Then $x(j+1)$ is an ascent top
  and $x(j+1)=x(i)$ with $i<j+1$. Thus $x\notin\Modasc_n$ by
  Lemma~\ref{lemma_ascent_tops}.  Suppose that $x(i) x(i+1)$ is an
  occurrence of $\mathfrak{b}$ and let $k = x(i+1)$. Then $x(i+1)$ is
  the leftmost occurrence of $k$ in $x$, but $x(i+1)$ is not an ascent
  top. Again, $x\notin\Modasc_n$ by Lemma~\ref{lemma_ascent_tops}.

  Conversely, suppose that $x$ avoids both $\mathfrak{a}$ and
  $\mathfrak{b}$. We shall use the recursive definition of $\Modasc$ to
  prove that $x$ is a modified ascent sequence. Let
  $v=x(1) \cdots x(n-1)$ and let $a = x(n)$. Note that $v$ avoids
  $\mathfrak{a}$ and $\mathfrak{b}$, but $v$ is not necessarily a Cayley
  permutation. We distinguish the following three cases.
  \begin{itemize}
  \item If $x(n-1) > x(n)$, then $x(n)$ is not the leftmost occurrence of
    $a$ in $x$ (since $x$ avoids $\mathfrak{b}$). Thus $v$ is a Cayley
    permutation: it contains all the integers from $1$ to
    $\max(v)=\max(x)$. By the inductive hypothesis, $v$ is a modified
    ascent sequence. Since $x=v a$, with $1 \le a \le x(n-1)$, we have
    that $x$ is also a modified ascent sequence.

  \item If $x(n-1)=x(n)$, then $v$ is again a Cayley permutation and we
    can proceed as in the previous case.

  \item If $x(n-1) < x(n)$, then $x(n)$ must be the only occurrence of
    $a$ in $x$ (since $x$ avoids $\mathfrak{a}$). Because $x$ is a
    Cayley permutation, the string $w$ obtained from $v$ by decreasing
    each entry $c>a$ by one must also be a Cayley permutation (that
    still avoids $\mathfrak{a}$ and $\mathfrak{b}$). By the inductive
    hypothesis, $w$ is a modified ascent sequence and
    $x(n) \le \max(w)+1 = \asc(w)+2$. Therefore $x$ is a modified
    sequence (since $x=\tilde{w}x(n)$ with
    $x(n-1)< x(n) \le \asc(w)+2$).
  \end{itemize}
  Finally, $\Modasc$ and the set of Fishburn permutations,
  $\F=\Cay(11,\fishpattern)$, are equinumerous. Therefore the two sets
  $\lbrace \mathfrak{a},\mathfrak{b} \rbrace$ and
  $\lbrace 11,\fishpattern \rbrace$ are Wilf-equivalent.
\end{proof}

\begin{lemma}\label{lemma_composition_is_sort}
  We have
  $X =\bigl\{ x \in \Cay: (\tilde{\upsilon}\circ\gamma)(x) =
  \sort(x)\bigr\}$.
\end{lemma}

\begin{proof}
  For any Cayley permutation $x$ we have
  $(\identity, x)^T=(\sort(x), \gamma(x))$. For any permutation $\pi$ we
  have
  $(\tilde{\upsilon}(\pi), \pi)^T=(\sort(\pi), \eta(\pi)) = (\identity,
  \eta(\pi))$. Thus
  \begin{align*}
    x\in X
    &\iff x = \eta(\pi)
    &&\text{for some $\pi\in S$}\\
    &\iff (\tilde{\upsilon}(\pi), \pi)^T= (\identity, x)
    &&\text{for some $\pi\in S$}\\
    &\iff (\tilde{\upsilon}(\pi), \pi)= (\sort(x), \gamma(x))
    &&\text{for some $\pi\in S$}\\
    &\iff \tilde{\upsilon}(\gamma(x)) = \sort(x). &&\qedhere
  \end{align*}
\end{proof}

\begin{theorem}\label{xset_mesh_patterns}
  We have $\xset = \Cay(\mathfrak{a},\mathfrak{c},\mathfrak{d})$, and
  hence the set $\lbrace \mathfrak{a},\mathfrak{c},\mathfrak{d} \rbrace$
  is Wilf-equivalent to the pattern $11$.
\end{theorem}

\begin{proof}
  Let $x\in\Cay_n$ be a Cayley permutation. We start by showing that if
  $x$ contains $\mathfrak{a}$, $\mathfrak{c}$ or $\mathfrak{d}$, then
  $x\notin\xset$. Let $\pi=\gamma(x)$. By
  Lemma~\ref{lemma_composition_is_sort}, it suffices to show that
  $\tilde{\upsilon}(\pi) \neq \sort(x)$. To ease notation, let
  $v = \tilde{\upsilon}(\pi)$.

  Suppose that $x(i) x(j) x(j+1)$ is an occurrence of the pattern
  $\mathfrak{a}$ in $x$. Then $x(i)=x(j+1)>x(j)$ and
  \begin{align*}
  \begin{pmatrix}
    \identity \\
    x
  \end{pmatrix}^{\!T}
    & =
  \begin{pmatrix}
    & \cdots & i & \cdots & j & j+1 & \cdots & \\
    & \cdots & x(i) & \cdots & x(j) & x(j+1) & \cdots & \\
  \end{pmatrix}^{\!T} \\
    &=
  \begin{pmatrix}
    & \cdots & x(j) & \cdots & x(j+1) & \cdots & x(i) & \cdots & \\
    & \cdots & j & \cdots & j+1 & \cdots & i & \cdots & \\
  \end{pmatrix}=
  \begin{pmatrix}
    \sort(x) \\
    \pi
  \end{pmatrix}.
  \end{align*}
  For $\ell\in [n]$, let $K(\ell)$ be the index of the column
  $(x(\ell), \ell)$ in $(\identity,x)^{\!T}$. In particular,
  $\sort(x)(K(\ell))=x(\ell)$. Note that $K(j)<K(j+1)<K(i)$.  In
  particular, the site in $\pi$ immediately after $K(j+1)$ is
  $\eta$-active. Therefore $v(K(i))>v(K(j+1))$, whereas
  $\sort(x)(K(i))=x(i)=x(j+1)=\sort(x)(K(j+1))$, and hence
  $\sort(x) \neq v$.

  Next, suppose that $x(i) x(j) x(j+1)$ is an occurrence of the pattern
  $\mathfrak{c}$ in $x$. Then $x(i)=x(j+1)+1<x(j)$ and
  \begin{align*}
    \begin{pmatrix}
      \identity \\ x
    \end{pmatrix}^{\!T}
    &=
      \begin{pmatrix}
        & \cdots & i & \cdots & j & j+1 & \cdots & \\
        & \cdots & x(i) & \cdots & x(j) & x(j+1) & \cdots & \\
      \end{pmatrix}^{\!T} \\
    &=
      \begin{pmatrix}
        & \cdots & x(j+1) & x(K(t)) & \cdots & x(i) & \cdots & x(j) & \cdots & \\
        & \cdots & j+1 & t & \cdots & i & \cdots & j & \cdots & \\
      \end{pmatrix}=
    \begin{pmatrix}
      \sort(x) \\
      \pi
    \end{pmatrix}.
  \end{align*}
  Note that $K(j+1)<K(i)<K(j)$. Now, if $K(i)=K(j+1)+1$, then $j+1>i$ is
  a descent in $\pi$ and, since $K(j)>K(j+1)$, the site in $\pi$
  immediately after $K(j+1)$ is not $\eta$-active. Therefore
  $v(K(i))=v(K(j+1))$, whereas
  $\sort(x)(K(i))=x(i)>x(j+1)=\sort(x)(K(j+1))$, and hence
  $\sort(x) \neq v$. Otherwise, consider the column $(x(t), t)$
  immediately after the column $(x(j+1), j+1)$ in $(\sort(x),\pi)$. In
  other words, suppose that $K(t)=K(j+1)+1$. Since $x(i)=x(j+1)+1$,
  either $x(t)=x(j+1)$ or $x(t)=x(j+1)+1$. Suppose that
  $x(t)=x(j+1)$. We shall prove by contradiction that
  $v(K(t)) \neq v(K(j+1))$, and thus $\sort(x) \neq v$. If
  $v(K(t))=v(K(j+1))$, then the site between $K(j+1)$ and $K(t)$ in
  $\pi$ is not $\eta$-active. Therefore $j+1>t$ is a descent. But then
  $x(t)=x(j+1)$ would precede $x(j+1)$ in $x$, which contradicts
  $x(i) x(j) x(j+1)$ being an occurrence of $\mathfrak{c}$ (since $x(t)$
  would be placed in a forbidden region). Finally, suppose that
  $x(t)=x(j+1)+1$. We wish to show that $v(K(t))=v(K(j+1))$. By
  contradiction, suppose that $v(K(j+1))<v(K(t))$; that is, the site
  between $K(j+1)$ and $K(t)$ is $\eta$-active. Since $K(j)>K(j+1)$, we
  have that $j+1<t$ is an ascent. But then $x(t)=x(j+1)+1 \le x(i)$,
  which contradicts $x(i)x(j)x(j+1)$ being an occurrence of
  $\mathfrak{c}$ (again $x(t)$ would be placed in a forbidden region).

  The pattern $\mathfrak{d}$ can be treated similarly, so we leave it to
  the reader.

  Conversely, suppose that $x$ avoids $\mathfrak{a}$, $\mathfrak{b}$ and
  $\mathfrak{c}$.  Let $\pi=\gamma(x)$. We wish to prove that
  $x=\eta(\pi)$ or, equivalently, $\sort(x)=\tilde{\upsilon}(\pi)$. Due
  to the great amount of technical details, we just sketch the proof. To
  prove the contrapositive statement, suppose that
  $\sort(x) \neq \tilde{\upsilon}(\pi)$. There are two possibilities for
  $\sort(x)$ to be different from the $\eta$-labeling of $\pi$. Either
  $\sort(x)$ labels two consecutive elements $\pi(i)$ with $k$ and
  $\pi(i+1)$ with $k+1$, but $i$ is not $\eta$-active. Or $\sort(x)$
  labels $\pi(i)$ and $\pi(i+1)$ with the same integer $k$, but the site
  $i$ is $\eta$-active. In the first case, $J(i)>i$ and
  $\pi(i)>\pi(i+1)$. If $J(i)=i+1$, then the labels $k$ of $\pi(i)$ and
  $k+1$ of $\pi(i+1)$ necessarily result in an occurrence of
  $\mathfrak{d}$ in $x$. Similarly, if $J(i)>i+1$, then the labels of
  $\pi(i)$, $\pi(i+1)$ and $\pi(J(i))$ result in an occurrence of
  $\mathfrak{c}$ in $x$. Analogously, if $\pi(i)$ and $\pi(i+1)$ are
  labeled with the same integer $k$, but the site $i$ is $\eta$-active,
  then it is possible to show that $x$ contains an occurrence of
  $\mathfrak{a}$. This completes the proof.
\end{proof}

Theorems~\ref{modasc_mesh_patterns} and \ref{xset_mesh_patterns}
characterize $\Modasc$ and $\xset$ as pattern avoiding Cayley
permutations. As a result, we can interpret the transports of patterns
described in Theorems~\ref{transport_theorem_modasc_fish} and
\ref{transport_theorem_xset_sym} as Wilf-equivalences.

\begin{corollary}
  Let $\sigma$ be a permutation.
  \begin{enumerate}
  \item The two sets $\lbrace 11,\fishpattern,\sigma \rbrace$ and
    $\lbrace \mathfrak{a},\mathfrak{b} \rbrace \cup\basis(\sigma)$ are
    Wilf-equivalent. That is,
    $$|\Cay_n(11,\fishpattern,\sigma)| \ = \ |\Cay_n(\mathfrak{a},\mathfrak{b},\basis(\sigma))|.
    $$
  \item The two sets $\lbrace 11,\sigma \rbrace$ and
    $\lbrace \mathfrak{a},\mathfrak{c},\mathfrak{d} \rbrace
    \cup\basis(\sigma)$ are Wilf-equivalent. That is,
    $$|\Cay_n(11,\sigma)| \ = \ |\Cay_n(\mathfrak{a},\mathfrak{c},\mathfrak{d},\basis(\sigma))|.
    $$
  \end{enumerate}
\end{corollary}

The transport between $\RGF$ and $\Sym(\dashpatt)$, that is
Theorem~\ref{transport_theorem_RGF}, can be interpreted as a
Wilf-equivalence on Cayley permutations as well.

\begin{lemma}
We have $\RGF=\Cay(\rgfpatt)$, where $\rgfpatt$ is the
Cayley-mesh pattern depicted in Figure~\ref{mesh_patterns_Sym0}.
\end{lemma}
\begin{proof}
A well known equivalent description of $\RGF$ is the following.
A Cayley permutation $x$ is a restricted growth function if and only
if every occurrence of each integer $k\ge 2$ in $x$ is preceded by
at least one occurrence of each of the integers $1,2,\dots,k-1$.
This property can in turn be expressed by the avoidance of
the Cayley-mesh pattern $\rgfpatt$.
\end{proof}

\begin{corollary}
Let $\sigma$ be a permutation. Then the two sets
$\lbrace 11,\dashpatt,\sigma\rbrace$ and
$\lbrace \rgfpatt\rbrace\cup[\sigma^{-1}]$ are Wilf-equivalent.
That is,
$$
|\Cay_n(11,\dashpatt,\sigma)| = |\Cay_n(\rgfpatt,[\sigma^{-1}])|.
$$
\end{corollary}

\subsection{Permutations with no $\eta$-inactive sites}\label{Section_S0}

Let $\Sym^0$ denote the set of permutations with no $\eta$-inactive
sites.  Note that if $\pi \in \Sym^0$, then
$\tilde{\upsilon}(\pi)= \identity$, and so $\eta(\pi)$ contains no
repeated letters. Indeed, $\eta(\pi) = \pi^{-1}$.  Thus
$\eta(\Sym^0)= (\Sym^0)^{-1}=\xset \cap \Sym$. When restricting to
permutations we can considerably simplify the mesh patterns
$\mathfrak{a}$, $\mathfrak{c}$ and $\mathfrak{d}$ that characterize
$\xset$: since the underlying pattern of $\mathfrak{a}$ is not a
permutation we can remove it; the pattern $\mathfrak{c}$ is
equivalent to the bivincular pattern
$\alpha=(231,\lbrace 2 \rbrace,\lbrace 1 \rbrace)$; and the pattern
$\mathfrak{d}$ is equivalent to the bivincular pattern
$\beta=(21,\lbrace 1 \rbrace,\lbrace 1 \rbrace)$. Thus
$$\eta(\Sym^0) = \Sym(\alpha,\beta).
$$
The patterns $\alpha$ and $\beta$ are depicted in
Figure~\ref{mesh_patterns_Sym0}.

\begin{figure}
  $$
  \rgfpatt \,=\,
  \begin{tikzpicture}[scale=0.50, baseline=19pt]
	\fill[NE-lines] (0,0.85) rectangle (0.85,1.15);
	\draw [semithick] (0,0.85) -- (3,0.85);
	\draw [semithick] (0,1.15) -- (3,1.15);
	\draw [semithick] (0,1.85) -- (3,1.85);
	\draw [semithick] (0,2.15) -- (3,2.15);
	\draw [semithick] (0.85,0) -- (0.85,3);
	\draw [semithick] (1.15,0) -- (1.15,3);
	\draw [semithick] (1.85,0) -- (1.85,3);
	\draw [semithick] (2.15,0) -- (2.15,3);
	\filldraw (1,2) circle (5pt);
	\filldraw (2,1) circle (5pt);
  \end{tikzpicture}
	\qquad
  \alpha \;=\;
  \begin{tikzpicture}[scale=0.4, baseline=20pt]
    \fill[NE-lines] (0,1) rectangle (4,2);
    \fill[NE-lines] (2,0) rectangle (3,4);
    \draw [semithick] (0,1) -- (4,1);
    \draw [semithick] (0,2) -- (4,2);
    \draw [semithick] (0,3) -- (4,3);
    \draw [semithick] (1,0) -- (1,4);
    \draw [semithick] (2,0) -- (2,4);
    \draw [semithick] (3,0) -- (3,4);
    \filldraw (1,2) circle (6pt);
    \filldraw (2,3) circle (6pt);
    \filldraw (3,1) circle (6pt);
  \end{tikzpicture}
  \qquad
  \beta \;=\;
  \begin{tikzpicture}[scale=0.4, baseline=15pt]
    \fill[NE-lines] (1,0) rectangle (2,3);
    \fill[NE-lines] (0,1) rectangle (3,2);
    \draw [semithick] (0,1) -- (3,1);
    \draw [semithick] (0,2) -- (3,2);
    \draw [semithick] (1,0) -- (1,3);
    \draw [semithick] (2,0) -- (2,3);
    \filldraw (1,2) circle (6pt);
    \filldraw (2,1) circle (6pt);
  \end{tikzpicture}
  \qquad
  \alpha^r \;=\;
  \begin{tikzpicture}[scale=0.4, baseline=20pt]
    \fill[NE-lines] (0,1) rectangle (4,2);
    \fill[NE-lines] (1,0) rectangle (2,4);
    \draw [semithick] (0,1) -- (4,1);
    \draw [semithick] (0,2) -- (4,2);
    \draw [semithick] (0,3) -- (4,3);
    \draw [semithick] (1,0) -- (1,4);
    \draw [semithick] (2,0) -- (2,4);
    \draw [semithick] (3,0) -- (3,4);
    \filldraw (1,1) circle (6pt);
    \filldraw (2,3) circle (6pt);
    \filldraw (3,2) circle (6pt);
  \end{tikzpicture}
  $$
  \caption{
  Mesh pattern such that $\RGF=\Cay(\rgfpatt)$ and
  bivincular patterns such that $\eta(\Sym^0)=\Sym(\alpha,\beta)$}\label{mesh_patterns_Sym0}
\end{figure}

We wish to construct a bijection between $\eta(\Sym^0)$ and the set of
ascent sequences with no flat steps (consecutive equal entries). An
ascent sequence with no flat steps is said to be
\emph{primitive}. Primitive ascent sequences were enumerated by
Dukes~\etal~\cite{DKRS}. Dukes and Parviainen~\cite{DP} proved that
primitive ascent sequences are in bijection with binary upper triangular
matrices with non-negative entries such that all rows and columns
contain at least one nonzero entry. The pattern $\alpha$ is closely
related to the Fishburn pattern $\fishpattern$. Let
$\alpha^r=(132,\lbrace 1 \rbrace,\lbrace 1 \rbrace)$, the reverse of
$\alpha$ (see Figure~\ref{mesh_patterns_Sym0}). Recall from
Section~\ref{Section_fishburn_perm} the step-wise procedure that
associates each Fishburn permutation $\pi$ with an ascent sequence
through the construction of $\pi$ from $1$ by inserting a new maximum, at
each step, and recording its position.  Parviainen~\cite{P} observed
that an alternative description of $\Ascseq$ can be obtained by
performing the same construction on $\Sym(\alpha^r)$ instead of
$\F$. The avoidance of $\alpha^r$ gives rise to an analogous notion of
$\alpha^r$-active site and the resulting bijection
$\psi': \Modasc\to\Sym(\alpha^r)$ can be computed using the Burge
transpose by replacing $\fishpattern$-active sites with
$\alpha^r$-active sites.

\begin{lemma}\label{lemma_flat_step}
  Let $x$ be a modified ascent sequence and let $\pi=\psi'(x)$. Then
  $\pi$ contains an occurrence of $\beta^r$ if and only if $x$ contains a
  flat step.
\end{lemma}
\begin{proof}
  Suppose that $\pi(i) \pi(i+1)$ is an occurrence
  of $\beta^r$ in $\pi$, or, equivalently, that $\pi(i+1)=\pi(i)+1$. Note that
  the site between $\pi(i)$ and $\pi(i+1)$ is not $\alpha^r$-active, since inserting a
  new maximum $n+1$ in this position would create an occurrence
  $\pi(i),n+1,\pi(i+1)$ of $\alpha^r$. Therefore the labels of $\pi(i)$ and
  $\pi(i+1)$ are equal. Since $\pi(i+1)=\pi(i)+1$, this results in a flat step
  $x(\pi(i))\,x(\pi(i+1))$ in $x$.  Conversely, suppose that $x(i)x(i+1)$ is
  a flat step in $x$. Then, by definition of Burge transpose, the
  elements $i+1$ and $i$ are in consecutive positions in $\pi$, and $i+1$
  precedes $i$. Thus $i+1,i$ is an occurrence of $\beta$, as desired.
\end{proof}

As a consequence of the proof of Lemma~\ref{lemma_flat_step}, $\psi'$ is
a bijection between the set of modified ascent sequences with no flat
steps and $\Sym(\alpha^r,\beta^r)$. Moreover, $\pi\mapsto\pi^r$ is a
bijection between $\Sym(\alpha^r,\beta^r)$ and
$\Sym(\alpha,\beta)=\eta(\Sym^0)$. Finally, since flat steps are
preserved when mapping a modified ascent sequence to its corresponding
ascent sequence, we obtain by composition the desired bijection between
$\eta(\Sym^0)$ and the set of primitive ascent sequences.
We close this section by stating this as a theorem.

\begin{theorem}
  There is a one-to-one correspondence between permutations with no
  $\eta$-inactive sites and the set of primitive ascent sequences.
\end{theorem}

\section{Future directions}\label{Section_future_works}

In this paper we have laid the theoretical foundations for the
development of a theory of transport of patterns from Fishburn
permutations to ascent sequences, and more generally between $\Sym$ and
$[\Cay]$, leaving most applications for future work. Given a
set of pattern avoiding Fishburn permutations, we have provided a
construction for the basis of the corresponding set of
modified ascent sequences. Using the bijection $\Modasc \to \Ascseq$, this
result can be interpreted in terms of (plain) ascent
sequences. Nevertheless, a more direct construction for
a basis would be of interest.

\begin{openproblem}
  Given a permutation $\sigma$, determine a set of Cayley permutations
  $\cbasis(\sigma)$ such that
  $$
  \fun(\Ascseq(\cbasis(\sigma)))=\F(\sigma).
  $$
\end{openproblem}

To find analogous sets in the other direction also remains an open problem.

\begin{openproblem}
  Given a Cayley permutation $x$, determine a set $\basis'(x)$ such that
  $$\psi(\Modasc(x)) = \F(\basis'(x)),$$
  and a set $\cbasis'(x)$ such that
  $$\varphi(\Ascseq(x)) = \F(\cbasis'(x)).$$
\end{openproblem}

Understanding how the avoidance of a pattern on ascent
sequences affects the corresponding set of modified ascent sequences, and
vice versa, seems to be necessary if we want to answer these
questions. In other words, we would like to describe the set of
sequences obtained by modifying $\Ascseq(x)$ in terms of avoidance of
patterns, as well as the set obtained by applying the inverse
construction to $\Modasc(x)$.

Our work suggests that a natural setting for the
transport of patterns is the set of Cayley permutations. Indeed, we
showed how a transport theorem often can be regarded as an example of
Wilf-equivalence over Cayley permutations. On the other hand, not all the ascent sequences
are Cayley permutations. This raises at least two more questions. First,
is there an analogue of the Burge transpose that allows us to
incorporate (plain) ascent sequences in the same framework? Secondly,
what natural superset $Y$ do ascent sequences
belong to? Ideally, since we would like to transport patterns between
$\Sym$ and $Y$, the set $Y$ should be equinumerous
with the set of permutations. A reasonable guess could then be
the set of inversion sequences, which properly contains $\Ascseq$, but
this remains to be investigated.

\end{document}